\documentclass{head/cernrep}

\usepackage{bbm}
\usepackage{xcolor}
\usepackage{mathrsfs}
\usepackage{amsthm}
\usepackage{natbib}
\usepackage{mathtools}
\usepackage{enumerate}
\usepackage{hyperref}
\usepackage{comment}
\usepackage{ulem}

\newcommand{\dif}{\mathrm{d}}
\newcommand{\bP}{\mathbb{P}}
\newcommand{\bE}{\mathbb{E}}
\newcommand{\bG}{\mathbb{G}}
\newcommand{\bR}{\mathbb{R}}

\newcommand{\bS}{\mathbb{S}}
\newcommand{\bN}{\mathbb{N}}
\newcommand{\eps}{\varepsilon}

\newcommand{\ind}{\mathbf{1}}

\newcommand{\cG}{\mathcal{G}}
\newcommand{\cF}{\mathcal{F}}
\newcommand{\cS}{\mathcal{S}}

\newcommand{\Dkonv}{\stackrel{\mathcal{D}}{\longrightarrow}}

\theoremstyle{remark}
\newtheorem{remark}{Remark}[section]
\theoremstyle{plain}
\newtheorem{proposition}{Proposition}[section]
\theoremstyle{plain}
\newtheorem{lemma}{Lemma}[section]
\theoremstyle{plain}
\newtheorem{theorem}{Theorem}[section]
\theoremstyle{plain}
\newtheorem{corollary}{Corollary}[section]
\theoremstyle{definition}

\begin{document}

\title{Adaptivity of the NPMLE to finitely discrete mixing distributions in Gaussian/Poisson mixtures}
\author{Yan Zhang, Stanislav Volgushev}
\institute{University of Toronto}

\begin{abstract}
We study the nonparametric maximum likelihood estimator (NPMLE) for Gaussian and Poisson mixture models, assuming the support of the true mixing distribution lies in a fixed bounded set. In this setting, we establish  exact parametric rates for both, marginal density estimation and the posterior mean when the true mixing distribution is finitely discrete. Moreover, we show that the NPMLE attains the optimal demixing rate previously known for overparameterized finite mixture models. Finally, we identify a new adaptivity phenomenon for inference: the likelihood ratio test statistic is asymptotically tight if and only if the true mixing distribution is finitely discrete.
\end{abstract}

\maketitle

\section{Introduction}

In recent years, the nonparametric maximum likelihood estimator (NPMLE) for mixture models has emerged as a powerful tool across a broad range of tasks, including Gaussian denoising \citep{saha2020nonparametric}, multiple hypothesis testing \citep{ignatiadis2025empirical}, and natural language modeling \citep{han2025besting}; see also \cite{koenker_gu_2026}. Among mixture models, Gaussian location mixtures have been studied most extensively \citep{zhang2009generalized,jiang2009general,saha2020nonparametric,soloff2025multivariate}. Much of the existing literature works in a fully nonparametric regime, imposing minimal structure on the mixing distribution and, in particular, not presuming bounded support. In that setting, the convergence rates for density estimation and empirical Bayes procedures can adapt to the tail behavior of the true mixing distribution. In this paper, we consider a complementary regime in which the support of the mixing distribution is known a priori to lie in a fixed bounded set. For Gaussian and Poisson mixture models under this assumption, we establish new adaptivity phenomena for the NPMLE.

To formalize the model class, let \(\Theta \subseteq \mathbb{R}^d\) be the parameter space and consider a family of densities \(\{p_\theta:\theta\in\Theta\}\) on \(\mathbb{R}^p\), each defined with respect to a common reference measure \(\mu\). Let \(\mathcal G\) denote the set of all probability measures supported on \(\Theta\). For \(g\in\mathcal G\), define the corresponding mixture density
\[
f_g(x)\coloneqq \int_{\Theta} p_\theta(x)\,\mathrm{d}g(\theta).
\]
The measure \(g\) is referred to as the \textit{mixing distribution}. We observe an i.i.d.\ sample \(X_1,\ldots,X_n\sim f_{g_0}\) for some \(g_0\in\mathcal G\). The NPMLE of the mixing distribution is any maximizer
\[
\hat g \in \arg\max_{g\in\mathcal G}\,\ell_n(f_g),
\]
where \(\ell_n(f)\coloneqq \sum_{i=1}^n \log f(X_i)\) is the log-likelihood. In practice, the exact NPMLE is not computable and approximate maximizers of the likelihood need to be considered instead. A natural estimator of the marginal density $f_{g_0}$ and the posterior mean
\[
\bE_{g_0}[\theta| x]\coloneqq
\frac{\int_\Theta \theta\, p_{\theta}(x)\,\dif g_0(\theta)}{f_{g_0}(x)}.
\]
can be obtained from the plug-in principle.

In the case where \(p_\theta\) is Gaussian with location parameter $\theta$ and known, fixed variance, the first explicit adaptive-rate results for the NPMLE in the univariate location model appear in \citet{zhang2009generalized}. Under \(\Theta=\mathbb{R}\), that work showed that the plug-in estimator \(f_{\hat g}\) converges to the true marginal density \(f_{g_0}\) at a rate that depends on (and effectively adapts to) the tail behavior of \(g_0\). The associated adaptive empirical Bayes theory was developed in parallel by \citet{jiang2009general}. These results were extended to the multivariate Gaussian case by \citet{saha2020nonparametric}. A further insight of \citet{saha2020nonparametric} is that when \(g_0\) is supported on a bounded set \(S\), the NPMLE can also adapt to the ``size'' of \(S\). In particular, when \(S\) is discrete with \(K\) support points, the estimation error is nearly parametric---on the order of \(K/n\) up to polylogarithmic factors. Determining whether the logarithmic factor is necessary has remained open. We show that it can be removed when \(\Theta\) is known and bounded. 

When \(p_\theta\) is Poisson, the adaptivity story has largely centered on empirical Bayes estimation and, in particular, on the comparison between \(g\)-modeling via the NPMLE and \(f\)-modeling via the classical Robbins estimator \citep{polyanskiy2021sharp,shen2022empirical,jana2025optimal}. In the light-tail regime---for instance, when the true mixing distribution is compactly supported or subexponential---both approaches attain the minimax regret rate. In contrast, in heavier-tailed regimes where \(g_0\) is only assumed to have a bounded moment of some order, the NPMLE achieves the minimax rate up to logarithmic factors, whereas \(f\)-modeling can be polynomially suboptimal. Related phenomena also appear for marginal density estimation as a by-product of these works. Interestingly, we focus on an extremely light-tail regime in which \(g_0\) is finitely discrete, and show that the NPMLE again exhibits superior performance. For the Poisson case, another intriguing result showing parametric properties of the NPMLE in a ``local'' sense was established in~\citet{lambert1984asymptotic}, see Remark~\ref{rem:poisson} for more details. 

Despite the possibility of achieving nearly parametric performance for density estimation and empirical Bayes procedures, recovering the mixing distribution itself (i.e., demixing) is an intrinsically difficult problem, so even the NPMLE cannot generally be expected to perform well for this task. 
In the Gaussian location model, \citet{dedecker2013minimax} showed that the minimax rate under Wasserstein loss is necessarily of logarithmic order, and \citet{soloff2025multivariate} established that the NPMLE attains this benchmark. A corresponding slow rate for Poisson mixtures was obtained by \citet{miao2024fisher}. This difficulty can be understood through the self-regularization phenomenon identified by \citet{polyanskiy2020self}. They showed that, for a broad class of mixture models, when the true mixing distribution has bounded support, the NPMLE typically has only \(O(\log n)\) support points. Consequently, if \(g_0\) is continuous (for instance, uniform on an interval), then the Wasserstein distance \(W_1(\hat g,g_0)\) cannot decay faster than this order. A natural conjecture is that faster demixing should be possible when \(g_0\) itself is finitely discrete. Such demixing adaptivity was obtained in \citet{soloff2025multivariate} in the special case where \(g_0\) is a point mass. In this work, under the additional assumption that \(\Theta\) is a known bounded set, we extend this phenomenon to general finitely discrete \(g_0\), and in the univariate case obtain a rate matching what is known to be achievable by overparameterized finite mixture models \citep{ho2016strong}.

Finally, beyond estimation, we also study a canonical inferential procedure: the likelihood ratio test. In its simplest form, the likelihood ratio statistic is\footnote{See Remark~\ref{rem:LRTfinite} for additional details.}
\[
L_n(\mathcal G,g_0)\coloneqq 2\{\ell_n(f_{\hat g})-\ell_n(f_{g_0})\}.
\]
In regular parametric models, \(L_n(\mathcal G,g_0)\) converges to a chi-square distribution, with degrees of freedom reflecting the model dimension; in this sense, it provides an empirical measure of model complexity. For mixture models, however, such Wilks-type behavior can fail. Since \citet{hartigan1985failure}, it has been known that even in the Gaussian location mixture with \(g_0\) a point mass, \(L_n(\mathcal G,g_0)\) diverges when \(\Theta\) is unbounded, precluding standard likelihood-ratio calibration. More recently, \citet{azais2009likelihood} showed that when \(\Theta\) is bounded, \(L_n(\mathcal G,g_0)\) converges to a tight limit for Gaussian, Poisson, and binomial mixtures, although their analysis still assumes that \(g_0\) is degenerate at a single point. Despite this restriction, their theorem is notable in that it identifies a setting where a genuinely nonparametric model exhibits parametric-like asymptotic behavior. Our approach is inspired by this line of work and differs from the nonasymptotic techniques that dominate much of the recent NPMLE literature. Building on these ideas, we extend the tightness result beyond the degenerate case to general finitely discrete \(g_0\), and uncover an adaptivity phenomenon for likelihood-ratio inference: \(L_n(\mathcal G,g_0)\) is asymptotically tight if and only if \(g_0\) is finitely discrete.

\section{Adaptivity phenomenon in Gaussian/Poisson mixture models}

Before presenting our main results, we introduce some additional notation. Throughout, we will use $\bN$ to denote the set of nonnegative integers including zero. We will also write $[d]$ for $\{1,\dots,d\}$. For the class of mixing distributions $\cG$ as defined in the introduction, we will often write $\cF \coloneqq \{f_g: g \in \cG\}$ for the class of the resulting marginal distributions. Before Section~\ref{sec:generality}, we focus on Gaussian and Poisson mixture models. To treat these models in a unified way, we impose the following assumption.

\begin{enumerate}
\renewcommand{\theenumi}{(GP)}
\renewcommand{\labelenumi}{\theenumi}
\item \label{(GP)} Fix \(b\in\{0,1,\dots,d\}\). The component densities admit the product form
\[
p_\theta(x)=\prod_{l=1}^d p_{\theta_l}(x_l),
\]
where, for \(1\le l\le b\),
\[
p_{\theta_l}(x_l)=\phi(x_l-\theta_l),
\qquad 
\phi(t)\coloneqq \frac{1}{\sqrt{2\pi}}\exp\!\Big(-\frac{t^2}{2}\Big),
\]
and, for \(b+1\le l\le d\),
\[
p_{\theta_l}(x_l)=\frac{\theta_l^{x_l}}{x_l!}\exp(-\theta_l).
\]
The parameter space \(\Theta\subseteq \bR^b\times(0,\infty)^{d-b}\) is bounded. The base measure \(\mu\) is the corresponding product of Lebesgue measure (on \(\bR^b\)) and counting measure (on \(\bN^{d-b}\)).
\end{enumerate}

The setup above includes multivariate Gaussian and Poisson mixture models as special cases. For these models, it is known that the \textit{unconstrained NPMLE}---corresponding to $\Theta=\bR^b \times (0,\infty)^{d-b}$---achieves a nearly parametric rate for marginal density estimation when the support of $g_0$ is bounded; see Theorem 2.1 of \citet{saha2020nonparametric} for the Gaussian case and Theorem 5 of \citet{jana2025optimal} for the Poisson case. 

In contrast, when $\Theta$ is bounded, we show that this likelihood-based principle exhibits a sharp adaptivity phenomenon. A salient departure from previous work is that our results are formulated in terms of the chi-square divergence, 
$$
\chi^2(f,f_0)\coloneqq\Big\|\frac{f}{f_0}-1\Big\|_{L_2(f_0\dif\mu)}^2. 
$$

Exact computation of the NPMLE is typically infeasible due to the infinite-dimensional nature of the model space. As a result, especially in the Gaussian setting, results are often stated for a broader class of estimators characterized by an approximate likelihood optimality condition \citep{jiang2009general,saha2020nonparametric,soloff2025multivariate}. Specifically, define
$$
\cG_n(c)=\{g\in\cG:\ell_n(f_g)\ge\ell_n(f_{g_0})+c\}\cup\{\hat{g}\},\footnote{The union with $\{\hat{g}\}$ guarantees that $\cG_n(c)\neq\emptyset$, even for $c>0$.}
$$
where $\hat{g}$ denotes the NPMLE introduced earlier, and let $c=c_n$ be a sequence of negative numbers that decreases slowly with $n$. Any estimator $\tilde{g}\in\cG_n(c_n)$ inherits strong statistical performance, with the NPMLE as a particular instance. 

\begin{theorem}\label{thm:MAIN}
Suppose that Assumption \ref{(GP)} holds.
\begin{enumerate}
\item If \(g_0\) is discrete with finitely many support points, then for any \(c<0\),
\[
\sup_{g\in\cG_n(c)} n\,\chi^2(f_g,f_{g_0})=O_\bP(1).
\]
\item If \(g_0\) is not finitely discrete\footnote{That is, \(g_0\) is not a discrete distribution supported on finitely many points.}, then for any \(c>0\),
\[
\sup_{g\in\cG_n(c)} n\,\chi^2(f_g,f_{g_0})\to\infty
\qquad\text{in probability.}
\]
\end{enumerate}
\end{theorem}

It is remarkable that the NPMLE can achieve exact parametric rates without any explicit regularization. This, together with other adaptivity results established in \cite{jiang2009general, saha2020nonparametric, soloff2025multivariate,shen2022empirical,jana2025optimal}, could be seen as theoretical explanation for the empirical observation that NPMLE-based methods work well in practice.

\begin{remark}[Adaptation of the NPMLE to finitely discrete \(g_0\)]\label{rem:adaptivity}
The first part of the theorem shows that, when \(g_0\) is finitely discrete and \(\Theta\) is bounded, the NPMLE attains an exact parametric rate for marginal density estimation. For univariate Gaussian and Poisson mixtures with $g_0$ corresponding to a single point mass, a result of this kind for the exact NPMLE is implicit in the proofs of~\cite{azais2009likelihood}. To the best of our knowledge, Theorem~\ref{thm:MAIN} is the first result that establishes exact parametric rate for density estimation based on the approximate NPMLE in a general setting. A different adaptivity result for multivariate Gaussian mixtures in the unconstrained setting (i.e., \(\Theta=\bR^d\)) is due to \citet{saha2020nonparametric}. The latter result is not directly comparable to ours. For clarity, we highlight several key differences:
\begin{itemize}
    \item \citet{saha2020nonparametric} formulate the rate in terms of the Hellinger distance,
    \[
    H^2(f,f_0)\coloneqq \Big\|\sqrt{\frac{f}{f_0}}-1\Big\|_{L_2(f_0\,\dif\mu)}^2.
    \]
    While the chi-square divergence dominates the Hellinger distance in general, it was recently shown that, for Gaussian location mixtures with bounded mixing support, the two divergences are in fact equivalent; see Theorem 21 in \citet{jia2023entropic}. We are not aware of an analogous equivalence result for Poisson mixtures.
    
    \item When \(g_0\) is finitely discrete, the bounds in \citet{saha2020nonparametric} do not yield an exact parametric rate. In contrast, our bounded-\(\Theta\) regime identifies a setting in which an exact parametric rate is attained. 
    
    \item On the technical side, \citet{saha2020nonparametric} provide nonasymptotic bounds that support a minimax analysis, whereas our result is purely asymptotic. In this sense, our theorem can be viewed as a benchmark for future nonasymptotic refinements.
\end{itemize}
\end{remark}

\begin{remark}[Literature in the bounded-\(\Theta\) regime]
The literature on the NPMLE for Gaussian and Poisson mixtures with a bounded parameter space \(\Theta\) is relatively sparse. For univariate Gaussian location mixtures, \citet{ma2025best} (Theorem 8) establishes an improved convergence rate for the NPMLE---relative to the unconstrained NPMLE---for a general mixing distribution \(g_0\) supported on \(\Theta\). For the univariate Poisson model, related arguments appear only implicitly in the proofs of \citet{jana2025optimal}. We conjecture that one reason this regime has attracted less attention is that its role in understanding adaptivity of the NPMLE, especially with respect to finitely discrete mixing distributions, has not been fully recognized. Nevertheless, bounded-support assumptions feature prominently in parts of the method-of-moments literature, where they enable parametric rates for marginal density estimation in finite mixture models; see, for example, \citet{wu2020optimal,doss2023optimal}.
\end{remark}

\begin{remark}[Nonparametric behavior when \(g_0\) is not finitely discrete]
The second part of the theorem shows that, when \(g_0\) is not finitely discrete, there exist estimators whose likelihood is large relative to that of the truth---in the sense that it exceeds \(\ell_n(f_{g_0})\) by a fixed positive amount---whose induced marginals do not converge at a parametric rate. Since only an approximate NPMLE can be computed in practice, this result indicates that one cannot expect computationally feasible estimators to achieve parametric rates for marginal density estimation in cases where the mixing distribution is not finitely discrete, even at fixed distributions in the parameter space. However, this statement does not preclude the possibility that the NPMLE itself may still achieve a parametric rate. Beyond its conceptual interest, the result has a practical implication for theory: for general \(g_0\), one cannot prove a parametric rate for the NPMLE using only the fact that it attains a likelihood larger than \(\ell_n(f_{g_0})\). In other words, a ``large-likelihood'' guarantee by itself is insufficient to certify parametric convergence in this case. 
\end{remark}

\begin{remark}[Other mixture models] 
The behavior of the NPMLE depends on the underlying model. For instance, mixtures of uniform distributions can be equivalently characterized through monotone density estimation; the corresponding estimator is the well-known Grenander estimator. There exists an extensive literature on the properties of this estimator, see \cite{groeneboom2015nonparametric} for an overview of existing results. The special structure of the Grenander estimator leads to completely different phenomena and proofs, but some of the results share a similar flavor to what we establish here. For instance, it is known that the marginal density is estimated at at $n^{1/3}$ pointwise rate when the mixing distribution has a density that is bounded away from zero, but $n^{1/2}$ rates are achieved at certain points when the mixing distribution has regions of zero density \citep{carolan1999asymptotic}.
\end{remark}

In the following subsection, we explain the intuition and implications behind the first part of Theorem~\ref{thm:MAIN}. The second part turns out to be more general; a broader, model-agnostic discussion is deferred to Section~\ref{sec:generality}.

\subsection{A general lemma for the parametric rate}\label{sec:parametric-lemma}

Despite recent advances in the theory of mixture models, the first ingredient underlying the parametric rate is in fact a general lemma, closely related versions of which appeared previously in the likelihood ratio testing literature; see, for example, \citet{gassiat2002likelihood,azais2009likelihood}. In this subsection, we denote by \(\cF\) a general class of densities with respect to a common reference measure \(\mu\), and assume an i.i.d. sample \(X_1,\dots,X_n\sim f_0\) for some \(f_0\in\cF\). The framework in Assumption~\ref{(GP)} can be viewed as a special case of this abstract setting.

On the pair \((\cF,f_0)\), we impose the following assumptions:
\begin{enumerate}
\renewcommand{\theenumi}{(A1)}
\renewcommand{\labelenumi}{\theenumi}
\item \label{(A1)} For every \(f\in \cF\backslash f_0\),
\[
\chi^2(f,f_0)\in(0,\infty).
\]

\renewcommand{\theenumi}{(A2)}
\renewcommand{\labelenumi}{\theenumi}
\item \label{(A2)} The \textit{score set}
\begin{equation}\label{eq:SCO}
\cS\coloneqq\bigl\{s_f:\ f\in \cF\backslash f_0\bigr\},
\qquad 
s_f\coloneqq \frac{f/f_0-1}{\chi(f,f_0)},
\end{equation}
is \(f_0\,\dif\mu\)-Donsker and admits an \(f_0\,\dif\mu\)-square-integrable envelope.
\end{enumerate}

Assumption~\ref{(A1)} serves mainly to ensure that the score set in \ref{(A2)} is well defined. Conditions of the form \ref{(A2)} have appeared repeatedly in the likelihood ratio testing literature; see, for instance, Theorem 3.1 of \citet{gassiat2002likelihood} and Theorem 3.1 of \citet{liu2003asymptotics}. Under these assumptions, we obtain the following lemma.

\begin{lemma}\label{lem:CVR}
Under Assumptions~\ref{(A1)} and \ref{(A2)},
\begin{equation}\label{eq:CVR}
\sup_{f\in\cF_n} n\,\chi^2(f,f_0)=O_{\bP}(1),
\end{equation}
where
\[
\cF_n \coloneqq \Big\{f\in\cF:\ \ell_n(f)\ge \ell_n(f_0)-c_n\Big\}
\]
for an arbitrary nonnegative sequence \(c_n=O_\bP(1)\).
\end{lemma}

Lemma~\ref{lem:CVR} can be viewed as a strengthening of Inequality 1.1 in \citet{gassiat2002likelihood}, which yields the same conclusion but only in the special case \(c_n\equiv 0\). While allowing a nonzero tolerance may seem minor, it is significant from a practical perspective in nonparametric mixture models. As discussed before Theorem~\ref{thm:MAIN}, the exact NPMLE is typically computationally infeasible, so any guarantees for the exact maximizer of the likelihood do not translate into guarantees on the performance of estimators that are used in practice. It is thus natural to seek guarantees for classes of tractable, approximately optimal estimators. 

In many implementations, the likelihood gap between the NPMLE and a candidate estimator can be evaluated; see, for example, Section 6.4 of \citet{lindsay1995mixture}. This motivates the class
\[
\tilde{\cF}_n \coloneqq \Big\{f\in\cF:\ \sup_{\tilde f\in\cF}\ell_n(\tilde f)-\ell_n(f)\le c_n\Big\}.
\]
Since \(\tilde{\cF}_n\subseteq\cF_n\), the bound \eqref{eq:CVR} immediately carries over to \(\tilde{\cF}_n\). By contrast, if one restricts to \(c_n\equiv 0\), then the parametric rate is guaranteed only for exact maximizers of \(\ell_n(f)\), a requirement that is typically computationally intractable.

Despite its concise statement, the assumptions of Lemma~\ref{lem:CVR} are notoriously difficult to verify for mixture models, largely because of Assumption~\ref{(A2)}. In the classical NPMLE literature \citep{ghosal2001entropies,zhang2009generalized,saha2020nonparametric}, the main technical step typically reduces to controlling the complexity (e.g., entropy numbers) of the marginal density class \(\cF\), or of classes of density derivatives in empirical Bayes analyses. In contrast, Assumption~\ref{(A2)} requires complexity control of the score set \(\cS\), which introduces additional layers of difficulty. First, the scores involve the likelihood ratio \(f/f_0\), so one generally cannot expect an entropy bound for \(\cS\) that is uniform in \(f_0\). In fact, results that we establish in later sections imply that for Gaussian and Poisson mixtures the score set is Donsker if and only if $f_0$ corresponds to a finitely discrete mixing distribution $g_0$. This shows that any universal entropy bound will lead to sub-optimal bounds, and the specific structure of $g_0$ needs to be taken into account to obtain sharper results. Second, the ratio \(f/f_0\) is normalized to have mean zero and variance one under \(f_0\,\dif\mu\). As a consequence, even when \(f_1\) and \(f_2\) are both close to \(f_0\), their normalized scores \(s_{f_1}\) and \(s_{f_2}\) can differ substantially. In effect, this normalization amplifies small neighborhoods of \(f_0\), making local control significantly more delicate.

\subsection{A key lemma in moment theory}

Throughout this subsection, for ease of exposition, we focus on univariate Gaussian location mixtures; that is, we assume \ref{(GP)} with \(b=d=1\). Even in this simplest setting, the assumptions of Lemma~\ref{lem:CVR} have previously been verified only in the degenerate case where \(g_0\) is a single point mass \citep{azais2009likelihood}. The main technical contribution of our work is twofold: we verify these assumptions for general finitely discrete \(g_0\), and we show that they fail when \(g_0\) is not finitely discrete. This subsection highlights the mechanism behind the positive result.

In Gaussian mixture problems, the analysis often reduces to exploiting the analytic structure of the Gaussian kernel \(\phi\). The likelihood ratio testing literature \citep{azais2009likelihood,jiang2019rate} leverages perhaps its most powerful feature: fixing \(\theta_0\in\Theta\), the likelihood ratio admits the Hermite expansion
\[
\frac{\phi(x-\theta)}{\phi(x-\theta_0)}-1
=\sum_{k=1}^\infty \frac{(\theta-\theta_0)^k}{k!}\,H_k(x-\theta_0),
\]
where \(\{H_k\}_{k\in\bN}\) are the (probabilists') Hermite polynomials, which satisfy the orthogonality relation
\[
\int H_k(x)H_{k'}(x)\,\dif\phi(x)=k!\,\ind_{\{k=k'\}},
\qquad k,k'\in\bN.
\]
Consequently, for any mixing distribution \(g\in\cG\), the marginal density \(f_g\) admits the representation
\[
f_g(x)
=\Big(1+\sum_{k=1}^\infty \frac{m_{k,g}}{k!}\,H_k(x-\theta_0)\Big)\phi(x-\theta_0),
\qquad 
m_{k,g}\coloneqq\int (\theta-\theta_0)^k\,\dif g(\theta).
\]
It follows that the score set can be written as
\[
\cS=\Bigg\{\sum_{k=1}^\infty 
\frac{m_{k,g}-m_{k,g_0}}{k!\,\chi(f_g,f_{g_0})}\,
H_k(x-\theta_0)\,
\frac{\phi(x-\theta_0)}{f_{g_0}(x)}
\,:\, g\in\cG\backslash g_0\Bigg\}. 
\]
In the special case \(g_0=\delta_{\theta_0}\), the ratio \(\phi(x-\theta_0)/f_{g_0}(x)\) equals \(1\), and the summands are orthogonal with respect to the data-generating distribution \(f_{g_0}(x)=\phi(x-\theta_0)\). As a result, controlling the size of \(\cS\) reduces to controlling the coefficient sequence, which is tractable via the explicit identity
\[
\frac{m_{k,g}-m_{k,g_0}}{k!\,\chi(f_g,f_{g_0})}
=\frac{m_{k,g}}{k!\,\sqrt{\sum_{j=1}^\infty m_{j,g}^2/j!}}.
\]

For a general finitely discrete \(g_0\), these convenient properties no longer come for free. Fortunately, one can recover an \(f_{g_0}\)-orthogonal decomposition by pulling a factor \(\sqrt{\phi(x-\theta_0)/f_{g_0}(x)}\) outside the infinite series: the remaining summands are then orthogonal with respect to \(f_{g_0}\). This manipulation is valid provided that \(\sqrt{\phi(x-\theta_0)/f_{g_0}(x)}\) is uniformly bounded, which can be ensured by choosing \(\theta_0\) to be a support point of \(g_0\).

The remaining difficulty is that the coefficients no longer admit a simple closed form. To address this, we establish the following moment comparison inequality. A related result---controlling higher-order moment differences by lower-order ones---appears as Lemma 10 in \citet{wu2020optimal}, where both \(g\) and \(g_0\) are assumed to be finitely discrete.

\begin{lemma}\label{lem:MCL}
Assume that \(\Theta\subseteq\bR\) is bounded. Let \(g_0\) be a finitely discrete distribution on \(\Theta\) with \(J\) support points. Fix an arbitrary \(g\) on \(\Theta\). Then, for any \(k>2J\),
\begin{equation}\label{eq:MCL}
|m_{k,g}-m_{k,g_0}|
\le
k\big(\sup_{\theta\in\Theta}|\theta-\theta_0|+1\big)^{2Jk}
\max_{j\in [2J]}|m_{j,g}-m_{j,g_0}|.
\end{equation}
\end{lemma}

We combine Lemma~\ref{lem:MCL} with the fact that the chi-square divergence \(\chi^2(f_g,f_{g_0})\) can be lower bounded by a function of the first \(2J\) moment differences (see Lemma~\ref{lem:DEN-multi}). Together, these bounds imply that the score set has a controllable degree of complexity.

The dependence on \(J\) in \eqref{eq:MCL} is necessarily unfavorable: as \(J\to\infty\), the coefficient on the right-hand side diverges. Thus, Lemma~\ref{lem:MCL} is useful precisely in the finitely discrete regime. Indeed, when \(g_0\) is not finitely discrete, one can construct a finitely discrete distribution \(g\) that matches the first several moments of \(g_0\) while differing substantially in higher-order moments---for instance, via Gaussian quadrature constructions. In this case, agreement of low-order moments places essentially no restriction on high-order moments, leading to a much richer score set. This provides intuition for why Lemma~\ref{lem:CVR} fails for such \(g_0\).

\subsection{Statistical consequences}

\subsubsection{Marginal linear functionals}

Under Assumption~\ref{(GP)}, we have established the parametric-rate bound
\[
\chi^2(f_{\hat g},f_{g_0})=O_\bP\Big(\frac{1}{n}\Big)
\]
whenever the true mixing distribution \(g_0\) is finitely discrete. This chi-square control immediately yields parametric accuracy for a broad class of linear functionals. Indeed, for any \(h\) that is square-integrable with respect to \(f_{g_0}\,\dif\mu\), the Cauchy--Schwarz inequality gives
\[
\begin{aligned}
\Big|\int h(x)f_{\hat g}(x)\,\dif\mu(x)-\int h(x)f_{g_0}(x)\,\dif\mu(x)\Big|^2
&=
\Big|\int h(x)\Big(\frac{f_{\hat g}(x)}{f_{g_0}(x)}-1\Big)f_{g_0}(x)\,\dif\mu(x)\Big|^2\\
&\le
\Big(\int h^2(x)f_{g_0}(x)\,\dif\mu(x)\Big)\;
\chi^2(f_{\hat g},f_{g_0}).
\end{aligned}
\]
Thus, unlike bounds stated only in Hellinger distance, our chi-square rate implies parametric convergence for plug-in estimators of any \(L_2(f_{g_0}\dif\mu)\) functional.

\begin{corollary}\label{cor:plugin}
Assume \ref{(GP)} and suppose \(g_0\) is finitely discrete. Then for any \(h\in L_2(f_{g_0}\dif\mu)\),
\[
\Big|\int h(x)f_{\hat g}(x)\,\dif\mu(x)-\int h(x)f_{g_0}(x)\,\dif\mu(x)\Big|
=O_\bP\Big(\frac{1}{\sqrt{n}}\Big).
\]
\end{corollary}

To estimate the marginal functional \(\int h(x)f_{g_0}(x)\,\dif\mu(x)\), a baseline approach is the empirical mean
\[
\frac{1}{n}\sum_{i=1}^n h(X_i),
\]
which attains the parametric rate under the same square-integrability condition. One might therefore expect the NPMLE plug-in estimator \(\int h(x)f_{\hat g}(x)\,\dif\mu(x)\) to perform at least as well, since it exploits the additional structural information that the samples arise from a mixture model. Corollary~\ref{cor:plugin} confirms this intuition in the finitely discrete regime. The following remark discusses a different setting in which a parametric rate for such functionals has also been established.

\begin{remark}[Parametric rates beyond finitely discrete \(g_0\)]\label{rem:poisson}
For univariate Poisson mixtures, \citet{lambert1984asymptotic} revealed an interesting phenomenon: for indicator test functions \(h(x)=\ind_{\{x=k\}}\), \(k\in\bN\), the NPMLE plug-in estimator is asymptotically equivalent to the empirical mean in the sense that
\[
\sqrt{n}\Bigg(\int h(x)f_{\hat g}(x)\,\dif\mu(x)-\frac{1}{n}\sum_{i=1}^n h(X_i)\Bigg)=o_\bP(1).
\]
In particular, the plug-in estimator not only achieves the parametric rate but also inherits asymptotic normality from the empirical mean. This phenomenon, however, relies on additional regularity: \citet{lambert1984asymptotic} assume that \(g_0\) has infinitely many support points and satisfies a positive lower bound in a neighborhood of zero. A related asymptotic equivalence result was also established in \citet{van1997asymptotic}. Whether an analogous result holds for Gaussian mixtures remains an open question.
\end{remark}

\begin{remark}[Parametric rate for moment estimation]\label{rem:moment}
Corollary~\ref{cor:plugin} has an immediate implication for moment estimation. In Gaussian and Poisson mixture models, it is well known that moments of the mixing distribution can be expressed as marginal linear functionals of the form \(\int h(x)f_g(x)\,\dif\mu(x)\); see, for example, \citet{morris1982natural,lindsay1989moment}. Consequently, the corresponding moments of the NPMLE \(\hat g\) converge to those of \(g_0\) at the parametric rate. This has important consequences in the following sections. 
\end{remark}

\subsubsection{Empirical Bayes estimation}

While Corollary~\ref{cor:plugin} provides parametric rates for a broad class of marginal linear functionals, the primary object of interest in Gaussian and Poisson mixture models is often the posterior mean
\[
\bE_{g_0}[\theta| x]\coloneqq
\frac{\int_\Theta \theta\, p_{\theta}(x)\,\dif g_0(\theta)}{f_{g_0}(x)}.
\]
Its NPMLE plug-in counterpart \(\bE_{\hat g}[\theta| x]\) is a canonical empirical Bayes rule and has been widely studied, with substantial success in both theory and practice \citep{saha2020nonparametric,han2025besting,koenker_gu_2026}. We show that the parametric chi-square control established above also yields parametric rates for estimating this posterior mean.

\begin{corollary}\label{cor:postmean}
Assume \ref{(GP)} and suppose \(g_0\) is finitely discrete. Then
\[
\int \big\|\bE_{\hat g}[\theta| x]-\bE_{g_0}[\theta| x]\big\|^2\, f_{g_0}(x)\,\dif\mu(x)
=O_\bP\Big(\frac{1}{n}\Big),
\]
where \(\|\cdot\|\) denotes the Euclidean norm on \(\bR^d\). Moreover, if \(b\in\{0,d\}\), then the parametric rate also holds pointwise:
\[
\big\|\bE_{\hat g}[\theta| x]-\bE_{g_0}[\theta| x]\big\|
=O_\bP\Big(\frac{1}{\sqrt{n}}\Big),
\qquad \text{for every admissible outcome } x.
\]
\end{corollary}

\begin{remark}[Connection to moment estimation]
To control the error in empirical Bayes estimation, we take an approach that differs from standard arguments. In one dimension,\footnote{See Proposition~\ref{pro:postmean} for the multivariate version.} we show that there exists an envelope function \(h\in L_2(f_{g_0}\dif\mu)\) such that
\[
|\bE_{\hat g}[\theta| x]-\bE_{g_0}[\theta| x]|
\le
h(x)\,\max_{k\in [2|\mathrm{supp}(g_0)|]}|m_{k,\hat{g}}-m_{k,g_0}|,
\]
where $|\mathrm{supp}(g_0)|$ denotes the number of support points of $g_0$. In this way, the parametric rate for posterior mean estimation is inherited from the parametric rate for moment estimation; see Remark~\ref{rem:moment}. Our argument also suggests a broader message: any procedure that delivers sufficiently accurate moment estimation automatically yields accurate empirical Bayes estimation through the same mechanism. This includes, for instance, the denoised method of moments estimator of \citet{wu2020optimal}.
\end{remark}

\begin{remark}[Differences in criteria from the existing literature]
As in our chi-square analysis, the present empirical Bayes results are not directly comparable to much of the existing literature. One key distinction is that we study the NPMLE computed over a bounded parameter space \(\Theta\). In addition, our analysis assumes the existence of a true mixing distribution \(g_0\), i.e., a random-effects model. In empirical Bayes problems without an underlying \(g_0\) (the ``deterministic'' or ``compound decision'' setting), the standard performance criterion is instead
\[
\frac{1}{n}\sum_{i=1}^n \big\|\bE_{\hat g}[\theta| X_i]-\bE_{g_n}[\theta| X_i]\big\|^2,
\]
where \(g_n\) denotes the empirical distribution of the latent parameters \(\theta_1,\dots,\theta_n\) associated with the observations. Extending our arguments to this setting would require nonasymptotic control of the score set in Lemma~\ref{lem:CVR}, which we leave for future work. 
\end{remark}

\subsubsection{Demixing}

We now turn to the fundamental question of how well \(\hat g\) recovers the mixing distribution \(g_0\). Since the NPMLE is typically discrete and its support need not align with that of \(g_0\), a natural discrepancy measure is the \(r\)-Wasserstein distance,
\[
\mathcal{W}_r(g_1,g_2)\coloneqq \inf_{(U,V)\in\Pi(g_1,g_2)} \big(\bE\|U-V\|^r\big)^{1/r},
\]
where \(g_1\) and \(g_2\) are probability measures on \(\bR^d\) with finite \(r\)-th moments, and \(\Pi(g_1,g_2)\) denotes the set of couplings of \(g_1\) and \(g_2\).

Recovering \(g_0\) is a demixing problem and is known to be intrinsically difficult: for general \(g_0\), only logarithmic convergence rates for the Wasserstein distance are possible in Gaussian mixtures \citep{dedecker2013minimax} and in Poisson mixtures \citep{miao2024fisher}. The following corollary shows, however, that—as in marginal density estimation and empirical Bayes estimation—the NPMLE can adapt to additional structure: when \(g_0\) is finitely discrete, \(\hat g\) achieves a substantially better performance that is similar to what was previously established for over-parametrized finite mixture models. 

\begin{corollary}\label{cor:wasserstein}
Assume \ref{(GP)} and suppose \(g_0\) is finitely discrete. Then
\[
\mathcal{W}_1(\hat g,g_0)
=O_\bP\Big(\frac{|\mathrm{supp}(\hat g)|}{n^{1/4}}\Big),
\]
where \(|\mathrm{supp}(\hat g)|\) denotes the number of support points of \(\hat g\). Moreover, when \(d=1\), the factor \(|\mathrm{supp}(\hat g)|\) can be removed.
\end{corollary}

Our result is closest in spirit to Theorem 13 of \citet{soloff2025multivariate}. That work covers multivariate Gaussian location mixtures and proves a comparable \(\mathcal{W}_2\) rate for the unconstrained NPMLE in the special case where \(g_0\) is a point mass. Our approach, in contrast, draws on ideas from the method-of-moments literature. In one dimension, it essentially relies on Proposition 5 of \citet{wu2020optimal}, which yields the inequality
\[
\mathcal{W}_1(\hat g,g_0)\le C\,\sqrt{\max_{k\in [2|\mathrm{supp}(g_0)|]}|m_{k,\hat{g}}-m_{k,g_0}|},
\]
where \(C\) is a constant depending on the structure of \(g_0\) and the diameter of \(\Theta\). As in the empirical Bayes case, the resulting \(n^{-1/4}\) rate is inherited from the parametric rate for moment estimation.

For the MLE in overparameterized finite mixture models, \citet{ho2016strong} establish a general \(\mathcal{W}_2\) convergence rate of order \(n^{-1/4}\) (up to logarithmic factors) under strong identifiability conditions. In particular, Corollary 4.1 of \citet{ho2016strong} specializes this bound to Gaussian location mixtures. Our result shows that the NPMLE can exhibit the same “parametric” \(n^{-1/4}\) behavior, at least in the one-dimensional Gaussian and Poisson settings.

When \(d>1\), the additional factor \(|\mathrm{supp}(\hat g)|\) appears as a compromise in our multivariate extension. Controlling the support size of the NPMLE is itself a challenging and largely open problem. A notable recent work, \citet{polyanskiy2020self}, shows that \(|\mathrm{supp}(\hat g)|=O_\bP(\log n)\) when the true mixing distribution \(g_0\) is sub-Gaussian---a condition that is automatically satisfied when \(g_0\) is finitely discrete. Their result, however, is currently limited to the one-dimensional setting. Our results provide further motivation for extending such support-size bounds to the multivariate case, since they would immediately yield sharper Wasserstein rates.

\section{Adaptivity phenomenon for the likelihood rate test}

Beyond the adaptivity of the NPMLE in estimation, the next theorem shows that the NPMLE can also exhibit an adaptivity phenomenon in inference. Recall the likelihood ratio statistic \(L_n(\cG,g_0)\) defined in the introduction. We have the following result.

\begin{theorem}\label{thm:MAIN2}
Suppose Assumption~\ref{(GP)} holds.
\begin{enumerate}
\item If \(g_0\) is discrete with finitely many support points, then for \(\cS\) defined in \eqref{eq:defS-multi} in the Appendix
\[
L_n(\cG,g_0) \Dkonv \sup_{s \in \cS} \bigl[(\bG(s))_+\bigr]^2,
\]
where \(x_+\coloneqq\max\{x,0\}\), and \(\bG\) is a centered Gaussian process indexed by \(\cS\) with covariance
\[
\bE[\bG(s_1)\bG(s_2)] = \bE[s_1(X)s_2(X)], 
\qquad X \sim f_{g_0}.
\]
\item If \(g_0\) is not finitely discrete, then \(L_n(\cG,g_0)\) diverges to infinity in probability.
\end{enumerate}
\end{theorem}

A similar convergence--divergence dichotomy has so far only been established in dimension $d=1$ in the special case where \(g_0\) is a point mass. In a seminal work on the failure of the likelihood ratio test for Gaussian location mixtures, \citet{hartigan1985failure} showed that \(L_n(\cG,g_0)\) diverges when \(\Theta\) is unbounded. Later, \citet{azais2009likelihood} proved that \(L_n(\cG,g_0)\) converges to a tight limit when \(\Theta\) is bounded, for both Gaussian and Poisson mixtures. These results suggest that, in the point-mass case, the ``effective dimension'' of the model class is finite if and only if \(\Theta\) is bounded. Our theorem sharpens this perspective: even with bounded \(\Theta\), the effective dimension is not determined solely by the model class, but also depends on the structure of the true mixing distribution \(g_0\).

\begin{remark}[An application to testing the number of components]\label{rem:LRTfinite}
Theorem~\ref{thm:MAIN2} has a direct application to testing
\[
H_0:\ g_0\in\cG_K
\qquad \text{versus} \qquad
H_1:\ g_0\in\cG\setminus \cG_K,
\]
where \(\cG_K\) denotes the set of discrete distributions on \(\Theta\) with at most \(K\) support points. A natural likelihood ratio test is based on the statistic
\[
2\big\{\ell_n(f_{\hat g})-\ell_n(f_{\hat g_K})\big\}
=
2\big\{\ell_n(f_{\hat g})-\ell_n(f_{g_0})\big\}
-
2\big\{\ell_n(f_{\hat g_K})-\ell_n(f_{g_0})\big\},
\]
where \(\hat g_K\) denotes the MLE over \(\cG_K\). An asymptotic characterization of the second term,
\(2\{\ell_n(f_{\hat g_K})-\ell_n(f_{g_0})\}\), was obtained in \citet{liu2003asymptotics, dacunha1997testing}. Theorem~\ref{thm:MAIN2} complements this by providing the asymptotic behavior of the first term,
\(2\{\ell_n(f_{\hat g})-\ell_n(f_{g_0})\}\), thereby yielding the asymptotic distribution of the full likelihood ratio statistic.
\end{remark}

\begin{remark}[Inferential properties of the NPMLE]\label{rem:infer}
Despite the extensive literature on estimation with the NPMLE, rigorous results on its inferential properties are surprisingly sparse. Two notable exceptions are as follows.
\begin{itemize}
\item As noted in Remark~\ref{rem:poisson}, for univariate Poisson mixtures, \citet{lambert1984asymptotic} established asymptotic normality for NPMLE plug-in estimators of marginal probabilities, under a class of mixing distributions \(g_0\) that is not finitely discrete in a neighborhood of zero.
\item Consider mixtures of step-function kernels, 
\[
p_\theta(x)=\frac{1}{\theta}\ind_{\{x\in[0,\theta]\}}.
\]
Mixtures of these kernels characterize the class of non-increasing densities on \((0,\infty)\). In this setting, \citet{groeneboom2015nonparametric} derived the asymptotic limit of a likelihood ratio statistic of the form
\[
2\big\{\ell_n(f_{\hat g})-\ell_n(f_{\hat g_0})\big\},
\]
where \(\hat g_0\) denotes the NPMLE under the additional constraint \(f_{\hat g_0}(x_0)=f_{g_0}(x_0)\) for some \(x_0>0\). Their analysis requires that \(f_{g_0}\) have a continuous, strictly negative derivative in a neighborhood of \(x_0\), a condition that in particular rules out finitely discrete \(g_0\).
\end{itemize}
In contrast to these results, Theorem~\ref{thm:MAIN2} shows that the likelihood ratio statistic we consider becomes substantially more tractable precisely when \(g_0\) is finitely discrete.
\end{remark}

\section{Generality of nonparametric phenomena}\label{sec:generality}

In this section, we develop general results describing the nonparametric behavior of mixture models when \(g_0\) is not finitely discrete. Our analysis requires only the following mild condition, whose key content is identifiability of the mixing distribution through the chi-square divergence between the induced marginals.

\begin{enumerate}
\renewcommand{\theenumi}{(ID)}
\renewcommand{\labelenumi}{\theenumi}
\item \label{(ID)} The parameter space \(\Theta\subseteq\bR^d\) is bounded, and for any \(g\in\cG\),
\[
\chi^2(f_g,f_{g_0})=0 \quad \Longleftrightarrow \quad g=g_0.
\]
\end{enumerate}

Assumption~\ref{(ID)} already suffices to yield the nonparametric phenomena observed in the previous sections.

\begin{theorem}\label{thm:diverge}
Assume \ref{(ID)} and suppose \(g_0\) is not finitely discrete. Then for any \(c>0\),
\[
\sup_{g\in\cG_n(c)} n\,\chi^2(f_g,f_{g_0})\to\infty
\qquad\text{in probability.}
\]
Moreover, \(L_n(\cG,g_0)\to\infty\) in probability.
\end{theorem}

A closely related nonparametric phenomenon was observed by \citet{hartigan1985failure} in his study of likelihood ratio tests for a seemingly ``parametric'' mixture model. 
Hartigan considered Gaussian mixtures of the form
\[
f_{\theta,t}(x)\coloneqq(1-t)\phi(x)+t\,\phi(x-\theta),
\]
with \(\theta\in\bR\) and \(t\in[0,1]\). For any \(K\ge 1\), he introduced the subclass
\[
\cF^{K}\coloneqq\bigl\{f_{\theta,t}:\ \theta\in\{\theta_1,\dots,\theta_K\},\ t\in[0,1]\bigr\}.
\]
Assuming the data are generated from the standard normal density \(\phi\), Hartigan showed that if the locations \(\{\theta_1,\dots,\theta_K\}\) are sufficiently well separated, then the likelihood ratio statistic \(\sup_{f \in \cF^K}\ell_n(f)-\ell_n(\phi)\) is asymptotically bounded below by a random variable \(\mathbb{L}_K\), and \(\mathbb{L}_K\to\infty\) as \(K\to\infty\). For completeness, Section~\ref{sec:divergence-unbounded} gives a formal statement of this divergence phenomenon for both the Gaussian and Poisson models. This divergence hinges crucially on the unboundedness of the location parameter space. In contrast, establishing divergence in our bounded-\(\Theta\) setting requires a different line of reasoning that explicitly exploits the structure of the null mixing distribution \(g_0\).

The key step in proving Theorem~\ref{thm:diverge} is to show that for any \(K\ge 1\) there exists a subclass \(\cG^{\le K}\subseteq \cG\) such that the corresponding likelihood ratio statistic satisfies
\[
L_n(\cG^{\le K},g_0) \stackrel{\mathcal{D}}{\longrightarrow} \chi^2(K).
\]
Since \(K\) can be taken arbitrarily large, divergence of \(L_n(\cG,g_0)\) in probability follows. The nonparametric rate for the chi-square divergence is established later through its connection with the likelihood ratio statistic (see Theorem~\ref{thm:ASY}). 

The construction of \(\cG^{\le K}\) is based on a multiplicative perturbation of \(g_0\) by a weighted sum of orthogonal polynomials, say \(\{q_k\}_{k\in\bN}\), that are associated with \(g_0\).\footnote{See Proposition~\ref{pro:ORT-multi} in the supplement for details.} Specifically, we set 
\begin{equation}\label{def:Gk}
\cG^{\le K}\coloneqq  \Big\{g\in\cG: \exists c \in \bR^K \text{ s.t. }
\dif g(\theta)=\Big(1+\sum_{k=1}^Kc_kq_k(\theta)\Big)\dif g_0(\theta)
\Big\}.    
\end{equation}
The corresponding score set is characterized in Lemma~\ref{lem:INF-multi}, where we show that it coincides with the set of normalized linear combinations of \(K\) linearly independent functions in \(L_2(f_{g_0}\dif\mu)\). This structure yields the \(\chi^2(K)\) limiting distribution.

\begin{remark}[Unbounded parameter space]
The boundedness of \(\Theta\) in Assumption~\ref{(ID)} is largely a technical convenience; one can often reduce to the bounded case via a reparameterization.\footnote{For instance, in one dimension, the transformation \(\theta \mapsto \arctan(\theta)\) maps \(\bR\) onto \((-\pi/2,\pi/2)\). This trick, however, does not circumvent the bounded-support assumption in \ref{(GP)}, since the reparameterization introduces singular points in the new parametrization, preventing Lemma~\ref{lem:NUM-multi} from going through.} 
\end{remark}

\begin{remark}[Finite-support component families]\label{rem:finite-support}
Theorem~\ref{thm:diverge} shows that the nonparametric phenomenon arises broadly for mixture models satisfying the identifiability condition \ref{(ID)}. Importantly, \ref{(ID)} can fail in settings where the component family being mixed consists of densities with only finitely many support points. As noted in \citet{liu2003asymptotics}, such models typically exhibit parametric behavior. See Section~\ref{sec:finsup} of the supplement for a result in this direction.
\end{remark}

\section{Summary and discussion}\label{sec:discussion}

This paper studies adaptivity phenomena of the NPMLE in Gaussian and Poisson mixture models under a bounded parameter space \(\Theta\). Our first main result, Theorem~\ref{thm:MAIN}, establishes a sharp dichotomy for likelihood-based estimators: when the true mixing distribution \(g_0\) is finitely discrete, any estimator whose likelihood is within a fixed tolerance of \(\ell_n(f_{g_0})\) achieves a parametric rate in chi-square divergence; in contrast, when \(g_0\) is not finitely discrete, a ``large-likelihood'' condition alone cannot certify parametric convergence. We then show that this parametric behavior extends beyond marginal density estimation to downstream tasks in the finitely discrete regime, including estimation of marginal linear functionals and empirical Bayes posterior means, as well as recovery of the mixing distribution in Wasserstein distance. Beyond estimation, we provide an inferential analogue in Theorem~\ref{thm:MAIN2}: the likelihood ratio statistic exhibits a convergence--divergence dichotomy depending on whether \(g_0\) is finitely discrete. Finally, we show that the nonparametric phenomena above arise under a fairly general identifiability condition in Theorem~\ref{thm:diverge}.

There are many interesting questions for future research. Our results on parametric rates heavily utilize the specific structure of Gaussian and Poisson distributions. It is unclear in what generality this type of parametric adaptation occurs. Extending the empirical Bayes guarantees to the compound decision setting where no true underlying $g_0$ exists is also of interest. Determining bounds on the number of support points of the NPMLE in the multivariate case is a hard open problem, any progress in this direction that would have direct implications on demixing convergence rates in the multivariate setting. Finally, characterizing the divergence rate of the likelihood ratio statistic in the case of bounded parameter spaces when $g_0$ is not finitely discrete remains an interesting open problem.

\bibliographystyle{apalike}
\bibliography{ref}

\section{Supplement} \label{sec:proof}

\subsection{Parametric behavior for distribution with a finite number of support points.} \label{sec:finsup}

Here we discuss the setting briefly mentioned in Remark~\ref{rem:finite-support}, where the component family \(\{p_\theta:\theta\in\Theta\}\) is supported on a finite sample space. This regime includes binomial and multinomial mixtures---where mixing is over success probabilities---as important special cases.

\begin{enumerate}
\renewcommand{\theenumi}{(DS)}
\renewcommand{\labelenumi}{\theenumi}
\item \label{(DS)} There exists \(K\in\bN\) such that each \(p_\theta\) is a density with respect to counting measure on \(\{1,\dots,K\}\).
\end{enumerate}

A related convergence result for the likelihood ratio statistic was proved as Theorem 3.2 in \citet{liu2003asymptotics} for more general discrete models. In contrast, we show that the completeness and continuous sample-path assumptions imposed there are not needed in our setting with a nonparametric mixing class \(\cG\).

\begin{theorem}\label{thm:converge2}
Let \(\Theta\subseteq\bR^d\) be arbitrary. Assume \ref{(DS)} and suppose that \(f_{g_0}\) has full support on \(\{1,\dots,K\}\). Then for any \(c<0\),
\[
\sup_{g\in\cG_n(c)} n\,\chi^2(f_g,f_{g_0})=O_\bP(1).
\]
Moreover, \(L_n(\cG,g_0)\) converges in distribution to a tight limit.
\end{theorem}

Compared to Gaussian or Poisson mixtures, models satisfying \ref{(DS)} admit a much more explicit control of model complexity. Indeed, the largest model on \(\{1,\dots,K\}\) is simply the class of all discrete distributions supported on \(\{1,\dots,K\}\). For this saturated model, the likelihood ratio statistic converges to a \(\chi^2(K-1)\) distribution, which therefore provides an upper bound on the limiting distribution of \(L_n(\cG,g_0)\).

\subsection{Nonparametric behavior with unbounded parameter space}\label{sec:divergence-unbounded}

Consider the Gaussian or Poisson case. When the parameter space $\Theta$ is unbounded, the following negative result holds.

\begin{proposition}\label{pro:divergence}
Assume~\ref{(GP)} and let $d=1$. Take $g_0$ to be a point mass. If $\Theta$ is unbounded, then for any $c>0$,
\[
\sup_{g\in\cG_n(c)} n\,\chi^2\!\bigl(f_g,f_{g_0}\bigr)\to \infty
\qquad\text{in probability.}
\]
Moreover, $L_n(\cG,g_0)\to \infty$ in probability.
\end{proposition}

In the Gaussian case, the divergence of the likelihood ratio statistic was first observed by \citet{hartigan1985failure}. The proposition above directly extends their finding to the Poisson setting and, moreover, to the divergence of the normalized chi-square distance.

\subsection{General asymptotic theory for star-shaped models}

In this section, we present a general result under a ``star-shaped'' assumption. It simplifies some prior works in this particular setting and is a core ingredient in the proofs for the results in the previous sections. 

We consider the abstract setup in Section~\ref{sec:parametric-lemma}. Besides (A1) and (A2), we impose one more assumption on the pair $(\mathcal{F},f_0)$: 
\begin{enumerate} \renewcommand{\theenumi}{(SS)}
\renewcommand{\labelenumi}{\theenumi}
\item \label{(SS)} For any $f\in\cF$, the convex combination $(1-t)f_0+tf$ remains in $\cF$ for all $t\in[0,1]$. 
\end{enumerate}

Assumptions like \ref{(A1)} and \ref{(A2)} were previously imposed by \citet{gassiat2002likelihood}, \citet{liu2003asymptotics} and \citet{azais2009likelihood} who studied the behavior of the likelihood ratio test under very general conditions. Compared to those works, our assumptions are stronger in that we require a certain star-shaped structure of $\cF$ in \ref{(SS)}. This assumption is satisfied in mixture models with nonparametric mixing distributions, but fails in many other examples such as finite mixtures. \ref{(SS)} is thus tailored to our specific setting.   

The following theorem is the main result of this section. It shows that the asymptotic behavior of both the normalized chi-square divergence and the likelihood ratio statistic is governed by a Gaussian process indexed by the score set. The statement is reminiscent of Theorem 3.1 in \citet{liu2003asymptotics}, which treats only the likelihood ratio statistic. Here, by imposing \ref{(SS)}, we can dispense with the completeness and continuous sample-path assumptions required in that work. 

\begin{theorem}\label{thm:ASY}
Under \ref{(A1)}, \ref{(A2)}, and \ref{(SS)}, we have
\begin{equation}\label{eq:ASY}
\begin{aligned}
n\,\chi^2(\hat f,f_0)
&=2\big\{\ell_n(\hat f)-\ell_n(f_0)\big\}+o_{\bP}(1)\\
&=\sup_{s\in\cS}\bigl[(\bG_n(s))_+\bigr]^2+o_{\bP}(1),
\end{aligned}
\end{equation}
where \(\bG_n\) is the empirical process
\[
\bG_n(s)\coloneqq\sqrt{n}\Big(\frac{1}{n}\sum_{i=1}^n s(X_i)-\bE[s(X_1)]\Big),
\]
and \(\hat f\) is any (approximate) MLE satisfying
\[
\sup_{f\in\cF}\ell_n(f)-\ell_n(\hat f)=o_{\bP}(1).
\]
\end{theorem}

A full proof is given in Section \ref{sec:proofASY}. The proof for the second equality in \eqref{eq:ASY} follows the general line of reasoning in \citet{gassiat2002likelihood,liu2003asymptotics,azais2009likelihood}; our use of \ref{(SS)} allows us to avoid more delicate discussions involving differentiability in quadratic mean and related conditions. By contrast, the first equality appears to be new: it provides a direct link between likelihood ratio theory and chi-square risk in marginal density estimation. This identity is a key ingredient in establishing the divergence statement in Theorem~\ref{thm:MAIN}.

\subsection{Details on the score set in Theorem~\ref{thm:MAIN2}} \label{sec:multS}

We now describe the score set \(\cS\). Fix a support point \(\theta_0\) of \(g_0\). Its characterization will rely on moment tensors \(\{m_{k,g}\}_{k\in\bN}\), which generalize univariate moments. For \(\theta\in\bR^d\), the tensor \(\theta^{\otimes k}\in(\bR^d)^{\otimes k}\) is the \(k\)-way array with entries
\[
(\theta^{\otimes k})_{i_1,\dots,i_k}\;=\;\prod_{j=1}^k \theta_{i_j}.
\]
With this notation, for \(g\in\cG\) we define
\[
m_{k,g}\coloneqq\int_\Theta (\theta-\theta_0)^{\otimes k}\,\dif g(\theta).
\]

Next, define the orthogonal polynomial family associated with \(p_{\theta_0}\) by
\begin{equation}\label{eq:polynomial}
q_\alpha(x)\coloneqq
\left.\frac{\partial^\alpha}{\partial\theta^\alpha}\Big(\frac{p_\theta(x)}{p_{\theta_0}(x)}\Big)\right|_{\theta=\theta_0}
=
\prod_{l=1}^d
\left.
\frac{\partial^{\alpha_l}}{\partial\theta_l^{\alpha_l}}
\Big(\frac{p_{\theta_l}(x_l)}{p_{\theta_{0,l}}(x_l)}\Big)
\right|_{\theta_l=\theta_{0,l}},
\qquad \alpha\in\bN^d.
\end{equation}
Each \(q_\alpha\) is a product of (probabilists') Hermite polynomials and Poisson--Charlier polynomials, and the collection \(\{q_\alpha\}_{\alpha\in\bN^d}\) is orthogonal in \(L_2(p_{\theta_0}\,\dif\mu)\).

With this notation, the score set \(\cS\) takes the form
\begin{equation}\label{eq:defS-multi}
\cS\coloneqq\Bigg\{
s(\cdot)=\sqrt{\frac{p_{\theta_0}(\cdot)}{f_{g_0}(\cdot)}}
\left(
\sum_{k=1}^\infty\ \sum_{|\alpha|=k}
\frac{m_{\alpha,g}-m_{\alpha,g_0}}{\alpha!\,\chi(f_g,f_{g_0})}\,
q_\alpha(\cdot)
\right)
\sqrt{\frac{p_{\theta_0}(\cdot)}{f_{g_0}(\cdot)}}
:\ g\in\cG\backslash g_0
\Bigg\}.
\end{equation}
Here, for a multi-index \(\alpha\in\bN^d\) and \(\theta\in\bR^d\), we write
\[
\theta^\alpha\coloneqq \prod_{l=1}^d \theta_l^{\alpha_l}, 
\qquad 
\alpha!\coloneqq \prod_{l=1}^d \alpha_l!, 
\qquad 
|\alpha|\coloneqq \sum_{l=1}^d \alpha_l,
\]
and define the corresponding scalar moments
\[
m_{\alpha,g}\coloneqq\int_\Theta (\theta-\theta_0)^\alpha\,g(\dif\theta).
\]
Note that \(m_{\alpha,g}\in\bR\) is an entry of the tensor \(m_{|\alpha|,g}\in(\bR^d)^{\otimes|\alpha|}\).

\subsection{Proofs of the main results under \ref{(GP)}}

For an order-\(k\) tensor \(T\in(\bR^d)^{\otimes k}\), let \(\|T\|_\infty\) denote its maximum absolute entry and \(\|T\|_F\) its Frobenius norm (the square root of the sum of squared entries). We also define the spectral norm by
\[
\|T\|_2 \coloneqq \sup_{c_1,\dots,c_k \in \mathbb{S}^{d-1}} \langle T, c_1 \otimes \cdots \otimes c_k\rangle,
\]
where \(\bS^{d-1}\) denotes the unit sphere in \(\bR^d\), and, for two tensors \(T,T'\), \(\langle T,T'\rangle\) denotes the inner product of their vectorized versions. These norms satisfy
\[
\|T\|_\infty \;\le\; \|T\|_2 \;\le\; \|T\|_F \;\le\; d^{k/2}\|T\|_\infty.
\]

A tensor \(T\) is symmetric if
\[
T_{j_1,\dots,j_k} = T_{j_{\pi(1)},\dots,j_{\pi(k)}}
\quad \text{for all } j_1,\dots,j_k \in [d] \text{ and all permutations \(\pi\) of }[k].
\]
In particular, the moment tensors introduced in Section~\ref{sec:multS} are symmetric. For symmetric tensors, a classical result due to Banach \citep{banach1938homogene,friedland2018nuclear} yields the sharper characterization
\begin{equation}\label{eq:SPE}
\|T\|_2 = \sup_{c \in \mathbb{S}^{d-1}} \bigl|\langle T, c^{\otimes k}\rangle\bigr|.
\end{equation}

Throughout this section, we adopt the notation from Section~\ref{sec:multS} and additionally define
\[
M \coloneqq \sup_{\theta \in \Theta} \|\theta - \theta_0\|,
\]
where \(\|\cdot\|\) denotes the Euclidean norm on \(\bR^d\). We assume \ref{(GP)}. Our main goal is to prove Theorem~\ref{thm:MAIN} and Theorem~\ref{thm:MAIN2}. The positive parts of both results hinge on controlling the size of the score set introduced in Section~\ref{sec:parametric-lemma}. The next few propositions and lemmas collect structural identities and bounds that will be used for this purpose.

By Theorem 4 and Corollary 1 of \citet{morris1982natural}, we have the following.

\begin{proposition}\label{pro:polynomial}
Recall the definition of \(q_\alpha\) in~\eqref{eq:polynomial}. For any \(\alpha,\alpha'\in\bN^d\),
\begin{enumerate}[(i)]
\item \(\displaystyle \int q_{\alpha}(x)\,q_{\alpha'}(x)\, p_{\theta_0}(x)\,\dif \mu(x)
= a_\alpha \,\alpha!\,\mathbf{1}_{\{\alpha=\alpha'\}}.\)
\item \(\displaystyle \int q_\alpha(x)\, p_{\theta}(x)\,\dif \mu(x)
= a_\alpha\,(\theta-\theta_0)^\alpha.\)
\end{enumerate}
Here,
\[
a_\alpha\coloneqq\prod_{l=1}^d V(\theta_{0,l})^{-\alpha_l},
\qquad 
V(\theta_{0,l})\coloneqq
\begin{cases}
1, & \text{if the \(l\)th marginal is Gaussian},\\
\theta_{0,l}, & \text{if the \(l\)th marginal is Poisson}.
\end{cases}
\]
\end{proposition}

As noted by \citet{azais2009likelihood}, the connection between the numerator of the score functions~\eqref{eq:SCO} and the moments \(\{m_{k,g}\}_{k\in\bN}\) can be obtained via a Taylor expansion.

\begin{lemma}\label{lem:NUM-multi}
Assume \ref{(GP)}. For any \(g\in\cG\),
\[
\frac{f_g(x)}{p_{\theta_0}(x)}-1
=\sum_{k=1}^\infty\sum_{|\alpha|=k}\frac{m_{\alpha,g}}{\alpha!}\,q_\alpha(x),
\]
and the series converges absolutely for any \(x\in\mathrm{supp}(f_{g_0})\).
\end{lemma}

\begin{proof}[Proof of Lemma~\ref{lem:NUM-multi}]
Under \ref{(GP)}, the map \(\theta\mapsto p_\theta(x)\) is a product of entire functions (each depending on a distinct coordinate of \(\theta\)). Hence \(\theta\mapsto p_\theta(x)/p_{\theta_0}(x)\) is entire, and its multivariate Taylor expansion around \(\theta_0\) takes the form
\[
\frac{p_\theta(x)}{p_{\theta_0}(x)}-1
=\sum_{k=1}^\infty\sum_{|\alpha|=k}\frac{(\theta-\theta_0)^\alpha}{\alpha!}\,q_\alpha(x),
\]
with absolute convergence for every \(\theta\in\bR^b\times(0,\infty)^{d-b}\); see Theorem 1.2.5 and Corollary 2.3.7 of \citet{krantz2001function}.
Moreover, for each multi-index \(\alpha\),
\[
\left|\frac{(\theta-\theta_0)^\alpha}{\alpha!}q_\alpha(x)\right|
\le \frac{M^{|\alpha|}}{\alpha!}\,|q_\alpha(x)|.
\]
The sum $\sum_{k=1}^\infty\sum_{|\alpha|=k} \tfrac{M^k}{\alpha!}|q_\alpha(x)|$ converges since the Taylor series converges absolutely at $\theta=\theta_0+M$. 
This provides an integrable dominating function in \(\theta\), allowing us to apply Fubini's theorem to interchange summation and integration. Consequently,
\[
\begin{aligned}
\frac{f_g(x)}{p_{\theta_0}(x)}-1
&=\int\left(\frac{p_\theta(x)}{p_{\theta_0}(x)}-1\right)\,\dif g(\theta)\\
&=\int\left(\sum_{k=1}^\infty\sum_{|\alpha|=k}\frac{(\theta-\theta_0)^\alpha}{\alpha!}\,q_\alpha(x)\right)\,\dif g(\theta)\\
&=\sum_{k=1}^\infty\sum_{|\alpha|=k}\frac{m_{\alpha,g}}{\alpha!}\,q_\alpha(x),
\end{aligned}
\]
as claimed.
\end{proof}

The next lemma relates the denominators of the score functions~\eqref{eq:SCO} to moment differences. A closely related bound for multivariate Gaussian location mixtures was obtained in Theorem 9 of \citet{bandeira2020optimal}.

\begin{lemma}\label{lem:DEN-multi}
Under \ref{(GP)}, for any \(g\in\cG\),
\begin{multline*}
\sup_{k\in\bN}\left(\min_{|\alpha|=k}\frac{a_\alpha^2}{\int q_\alpha^2(x)f_{g_0}(x)\,\dif\mu(x)}\right)\|m_{k,g}-m_{k,g_0}\|_\infty^2
\\
\le \chi^2(f_g,f_{g_0})
\le C_0\sum_{k=1}^\infty\left(\sup_{|\alpha|=k}a_\alpha\right)\frac{\|m_{k,g}-m_{k,g_0}\|_F^2}{k!},   
\end{multline*}
where
\[
C_0\coloneqq\sup_{x\in\mathrm{supp}(f_{g_0})}\frac{p_{\theta_0}(x)}{f_{g_0}(x)}<\infty.
\]
\end{lemma}

\begin{proof}[Proof of Lemma~\ref{lem:DEN-multi}] 
The finiteness \(C_0<\infty\) follows since \(\theta_0\) is a support point of \(g_0\). To derive the upper bound, observe that
\[
\begin{aligned}
\chi^2(f_g,f_{g_0})
&= \int\left(\frac{f_g(x)}{f_{g_0}(x)}-1\right)^2 f_{g_0}(x)\,\dif\mu(x)\\
&=\int\left(\frac{f_g(x)-f_{g_0}(x)}{p_{\theta_0}(x)}\right)^2\frac{p_{\theta_0}(x)}{f_{g_0}(x)}\, p_{\theta_0}(x)\,\dif\mu(x)\\
&\le C_0\int\left(\sum_{k=1}^\infty\sum_{|\alpha|=k}\frac{m_{\alpha,g}-m_{\alpha,g_0}}{\alpha!}\,q_\alpha(x)\right)^2 p_{\theta_0}(x)\,\dif\mu(x)\\
&\le C_0\sum_{k=1}^\infty\sum_{|\alpha|=k} a_\alpha\,\frac{(m_{\alpha,g}-m_{\alpha,g_0})^2}{\alpha!}\\
&\le C_0\sum_{k=1}^\infty\left(\sup_{|\alpha|=k}a_\alpha\right)\frac{\|m_{k,g}-m_{k,g_0}\|_F^2}{k!},
\end{aligned}
\]
where the third line uses Lemma~\ref{lem:NUM-multi} and the definition of \(C_0\).
The fourth line follows from Proposition~\ref{pro:polynomial}(i). Indeed, if the series in the fourth line is infinite then the bound is trivial. Otherwise, Proposition~\ref{pro:polynomial}(i) implies that the partial sums
\[
\left\{\sum_{k=1}^K\sum_{|\alpha|=k}\frac{m_{\alpha,g}-m_{\alpha,g_0}}{\alpha!}\,q_\alpha\right\}_{K\in\bN}
\]
form a Cauchy sequence in \(L_2(p_{\theta_0}\,\dif\mu)\), and for each fixed \(K\),
\[
\Big\|\sum_{k=1}^K\sum_{|\alpha|=k}\frac{m_{\alpha,g}-m_{\alpha,g_0}}{\alpha!}\,q_\alpha \Big\|_{L_2(p_{\theta_0}\,\dif\mu)}^2
=\sum_{k=1}^K\sum_{|\alpha|=k} a_\alpha\,\frac{(m_{\alpha,g}-m_{\alpha,g_0})^2}{\alpha!}.
\]
In this case, the inequality in the fourth line is in fact an equality. The last line uses the identity
\begin{equation}\label{eq:Frob}
\|m_{k,g}-m_{k,g_0}\|_F^2=\sum_{|\alpha|=k}\frac{k!}{\alpha!}\,(m_{\alpha,g}-m_{\alpha,g_0})^2.
\end{equation}

Next, we prove the lower bound. By Proposition~\ref{pro:polynomial}(ii) and the Cauchy--Schwarz inequality,
\[
\begin{aligned}
|a_\alpha|\,|m_{\alpha,g}-m_{\alpha,g_0}|
&=\left|\int q_\alpha(x)\, f_g(x)\,\dif\mu(x)-\int q_\alpha(x)\,f_{g_0}(x)\,\dif\mu(x)\right|\\
&=\left|\int q_\alpha(x)\left(\frac{f_g(x)}{f_{g_0}(x)}-1\right) f_{g_0}(x)\,\dif\mu(x)\right|\\
&\le \chi(f_g,f_{g_0})\sqrt{\int q_\alpha^2(x)\, f_{g_0}(x)\,\dif\mu(x)},
\end{aligned}
\]
for any \(\alpha\in\bN^d\).
\end{proof}

The moment-comparison lemma below plays a central role in our proofs and may also be of independent interest.

\begin{lemma}\label{lem:MCL-multi}
Let \(g_0\) be a finitely discrete distribution with \(J\) support points. Fix \(g\in\cG\) and define
\begin{equation}\label{eq:moment-difference}
\Delta_g\coloneqq\max_{k\in[2J]}\|m_{k,g}-m_{k,g_0}\|_2.
\end{equation}
Then, for any \(k>2J\),
\begin{equation}\label{eq:MCL-multi}
\|m_{k,g}-m_{k,g_0}\|_2\le k(M+1)^{2Jk}\Delta_g.
\end{equation}
\end{lemma}

\begin{proof}[Proof of Lemma~\ref{lem:MCL-multi}] 
This lemma is a direct generalization of Lemma~\ref{lem:MCL}. For any \(k>2J\), we have
\[
\begin{aligned}
\|m_{k,g}-m_{k,g_0}\|_2
&=\sup_{c\in\bS^{d-1}}\big|\langle m_{k,g}-m_{k,g_0},c^{\otimes k}\rangle\big|\\
&=\sup_{c\in\mathbb{S}^{d-1}}\left|\int\langle\theta-\theta_0,c\rangle^k\,\dif(g-g_0)(\theta)\right|\\
&\le k(M+1)^{2Jk}\sup_{c\in\mathbb{S}^{d-1}}\max_{j\in[2J]}\left|\int\langle\theta-\theta_0,c\rangle^j\,\dif(g-g_0)(\theta)\right|\\
&= k(M+1)^{2Jk}\Delta_g .
\end{aligned}
\]
Here, the first and last equalities follow from \eqref{eq:SPE}. The inequality follows by applying Lemma~\ref{lem:MCL} to the one-dimensional distributions of \(\langle\theta,c\rangle\) for \(\theta\sim g_0\) and \(\theta\sim g\), respectively. Note that if \(g_0\) has \(J\) support points, then the corresponding distribution of \(\langle\theta,c\rangle\) has at most \(J\) support points, and Lemma~\ref{lem:MCL} still applies when the number of support points is at most \(J\). Also, by the definition of \(M\), we have \(|\langle\theta,c\rangle-\langle\theta_0,c\rangle|\le M\) for all \(c\in\bS^{d-1}\) and \(\theta\in\Theta\).
\end{proof}

\bigskip

\begin{proof}[Proof of Theorem~\ref{thm:MAIN} and Theorem~\ref{thm:MAIN2}] 
Assume that \(g_0\) is finitely discrete. To prove the first parts of both theorems, we apply Lemma~\ref{lem:CVR} and Theorem~\ref{thm:ASY}. It therefore suffices to verify Assumptions \ref{(A1)}, \ref{(A2)}, and \ref{(SS)} for the pair \((\cF,f_{g_0})\). Assumption~\ref{(SS)} holds by definition. 

For Assumption~\ref{(A1)}, note that a direct bound yields \(\|m_{k,g}-m_{k,g_0}\|_F^2\le 4d^kM^{2k}\). Hence, by Lemma~\ref{lem:DEN-multi},
\[
\chi^2(f_g,f_{g_0})
\le 4C_0\sum_{k=1}^\infty\left(\sup_{|\alpha|=k}a_\alpha\right)\frac{d^kM^{2k}}{k!},
\]
which is finite under \ref{(GP)} (recall the form of \(a_\alpha\) in Proposition~\ref{pro:polynomial}). It remains to verify Assumption~\ref{(A2)}.

By Lemma~\ref{lem:NUM-multi}, for any \(g\in\cG\),
\[
\begin{aligned}
\frac{f_g}{f_{g_0}}-1
&=\frac{p_{\theta_0}}{f_{g_0}}\left(\frac{f_g-f_{g_0}}{p_{\theta_0}}\right)\\
&=\frac{p_{\theta_0}}{f_{g_0}}
\left(\sum_{k=1}^\infty\sum_{|\alpha|=k}\frac{m_{\alpha,g}-m_{\alpha,g_0}}{\alpha!}\,q_\alpha\right)\\
&=\sqrt{\frac{p_{\theta_0}}{f_{g_0}}}
\left(\sum_{k=1}^\infty\sum_{|\alpha|=k}\frac{m_{\alpha,g}-m_{\alpha,g_0}}{\alpha!}\,
q_\alpha\sqrt{\frac{p_{\theta_0}}{f_{g_0}}}\right).
\end{aligned}
\]
Consequently, the score class admits the representation
\[
\cS
=\left\{\sqrt{\frac{p_{\theta_0}}{f_{g_0}}}
\sum_{k=1}^\infty\sum_{|\alpha|=k} c_{\alpha,g}\,h_\alpha
:\ g\in\cG\backslash g_0\right\},
\]
where
\[
c_{\alpha,g}\coloneqq
\sqrt{\frac{a_\alpha (k+1)^d}{\alpha!}}\,
\frac{|\alpha|\,(m_{\alpha,g}-m_{\alpha,g_0})}{\chi(f_g,f_{g_0})},
\qquad
h_\alpha\coloneqq
\frac{q_\alpha}{|\alpha|\sqrt{a_\alpha (k+1)^d\alpha!}}\,
\sqrt{\frac{p_{\theta_0}}{f_{g_0}}}.
\]

Since \(\sqrt{p_{\theta_0}/f_{g_0}}\) is uniformly bounded (by Lemma~\ref{lem:DEN-multi}), Example 2.10.10 of \citet{van1996weak} implies that it suffices to verify the Donsker property and to identify a square-integrable envelope for the class
\[
\left\{\sum_{k=1}^\infty\sum_{|\alpha|=k} c_{\alpha,g}h_\alpha:\ g\in\cG\backslash g_0\right\}.
\]
By Theorem 2.13.2 of \citet{van1996weak}, this class is \(f_{g_0}\,\dif\mu\)-Donsker provided that:
\begin{enumerate}[(a)]
\item \(\{h_\alpha\}_{\alpha\in\bN^d}\) is an orthogonal sequence in \(L_2(f_{g_0}\dif\mu)\) and
\(\sum_{k=1}^\infty\sum_{|\alpha|=k}\|h_\alpha\|_{L_2(f_{g_0}\dif\mu)}^2<\infty\).
\item For any \(g\in\cG\backslash g_0\), \(\sum_{k=1}^\infty\sum_{|\alpha|=k} c_{\alpha,g}h_\alpha\) converges pointwise, and
\(\sup_{g\in\cG}\sum_{k=1}^\infty\sum_{|\alpha|=k} c_{\alpha,g}^2<\infty\).
\end{enumerate}
Condition (a) follows directly from the definition of \(h_\alpha\) and Proposition~\ref{pro:polynomial}(i). To verify (b), we invoke the lower bound in Lemma~\ref{lem:DEN-multi}, which gives
\[
\begin{aligned}
\chi^2(f_g,f_{g_0})
&\ge \max_{k\in[2J]}
\left(\min_{|\alpha|=k}\frac{a_\alpha^2}{\int q_\alpha^2(x)f_{g_0}(x)\,\dif\mu(x)}\right)
\|m_{k,g}-m_{k,g_0}\|_\infty^2\\
&\ge
\left(\min_{k\in[2J]}\frac{1}{d^k}\min_{|\alpha|=k}\frac{a_\alpha^2}{\int q_\alpha^2(x)f_{g_0}(x)\,\dif\mu(x)}\right)\Delta_g^2,
\end{aligned}
\]
where \(\Delta_g\) is defined as in Lemma~\ref{lem:MCL-multi}. Applying \eqref{eq:MCL-multi}, we obtain
\[
\begin{aligned}
\sum_{k=1}^\infty\sum_{|\alpha|=k} c_{\alpha,g}^2
&=\sum_{k=1}^\infty\sum_{|\alpha|=k}
\frac{a_\alpha (k+1)^d k^2(m_{\alpha,g}-m_{\alpha,g_0})^2}{\alpha!\,\chi^2(f_g,f_{g_0})}\\
&\le \sum_{k=1}^\infty\left(\max_{|\alpha|=k}a_\alpha\right)
\frac{(k+1)^d k^2\|m_{k,g}-m_{k,g_0}\|_F^2}{k!\,\chi^2(f_g,f_{g_0})}\\
&\le
\left(\max_{k\in[2J]}d^k\max_{|\alpha|=k}\frac{\int q_\alpha^2(x)f_{g_0}(x)\,\dif\mu(x)}{a_\alpha^2}\right)
\\
&\quad\times\sum_{k=1}^\infty\left(\max_{|\alpha|=k}a_\alpha\right)
\frac{(k+1)^d k^2 d^k\|m_{k,g}-m_{k,g_0}\|_2^2}{k!\,\Delta_g^2}\\
&\le C_1,
\end{aligned}
\]
where
\[
\begin{aligned}
C_1
&\coloneqq
\left(\max_{k\in [2J]}d^k\max_{|\alpha|=k}\frac{\int q_\alpha^2(x)f_{g_0}(x)\,\dif\mu(x)}{a_\alpha^2}\right)
\Bigg(
\sum_{k=1}^{2J}\left(\max_{|\alpha|=k}a_\alpha\right)\frac{(k+1)^d k^2 d^k}{k!}\\
&\hspace{4.7em}+
\sum_{k=2J+1}^{\infty}\left(\max_{|\alpha|=k}a_\alpha\right)\frac{(k+1)^d k^4 d^k}{k!}(M+1)^{4Jk}
\Bigg)
<\infty.
\end{aligned}
\]
Finally, by the Cauchy--Schwarz inequality,
\[
\left|\sum_{k=1}^\infty\sum_{|\alpha|=k} c_{\alpha,g}h_\alpha\right|
\le
\sqrt{\sum_{k=1}^\infty\sum_{|\alpha|=k} c_{\alpha,g}^2}\;
\sqrt{\sum_{k=1}^\infty\sum_{|\alpha|=k} h_\alpha^2}
\le
\sqrt{C_1}\,\sqrt{\sum_{k=1}^\infty\sum_{|\alpha|=k} h_\alpha^2}.
\]
This yields an \(f_{g_0}\,\dif\mu\)-square-integrable envelope for \(\cS\). Consequently, Assumption~\ref{(A2)} holds.

Assume that \(g_0\) is not finitely discrete. The second parts of Theorem~\ref{thm:MAIN} and Theorem~\ref{thm:MAIN2} follow from Theorem~\ref{thm:diverge}, once we note that Assumption~\ref{(ID)} holds for the model \ref{(GP)}. Indeed, fix \(g\in\cG\). If \(\chi^2(f_g,f_{g_0})=0\), then Lemma~\ref{lem:DEN-multi} implies that the moments of \(g\) coincide with those of \(g_0\). Since \(\Theta\) is bounded, the Hausdorff moment problem is determinate, and thus \(g=g_0\).
\end{proof}

\subsubsection{Proof of Corollary~\ref{cor:postmean}}

Classical empirical Bayes analyses often begin with Tweedie's formula. However, Tweedie's formula does not apply to the composite model~\ref{(GP)}. Instead, define the vector-valued function
\[
f_g^+(x)\coloneqq \int (\theta-\theta_0)\,p_\theta(x)\,\dif g(\theta).
\]
Then, for any \(g\in\cG\), the posterior mean admits the representation
\[
\bE_g[\theta| x]=\theta_0+\frac{f_g^+(x)}{f_g(x)}.
\]

To analyze this quantity, we introduce the (vector-valued) moment coefficients
\[
m_{\alpha,g}^+\coloneqq \int (\theta-\theta_0)\,(\theta-\theta_0)^\alpha\,\dif g(\theta)
\]
where we recall the notation $(\theta-\theta_0)^\alpha := \prod_{l=1}^d (\theta_l - \theta_{0,l})^{\alpha_l}$.
Analogously to Lemma~\ref{lem:NUM-multi}, we will use an infinite-series expansion for \(f_g^+\).

\begin{lemma}\label{lem:expansion2}
Assume \ref{(GP)}. For any \(g\in\cG\),
\[
f_g^+(x)
=
p_{\theta_0}(x)\sum_{k=0}^\infty\sum_{|\alpha|=k}\frac{m_{\alpha,g}^+}{\alpha!}\,q_\alpha(x),
\]
and the series converges absolutely for any \(x\in\mathrm{supp}(f_{g_0})\).
\end{lemma}

\begin{proof}[Proof of Lemma~\ref{lem:expansion2}] 
Fix \(l\in[d]\). Under \ref{(GP)}, the Taylor series
\[
(\theta_l-\theta_{0,l})\left(\frac{p_\theta(x)}{p_{\theta_0}(x)}-1\right)
=
\sum_{k=1}^\infty\sum_{|\alpha|=k}
\frac{(\theta_l-\theta_{0,l})(\theta-\theta_0)^\alpha}{\alpha!}\,q_\alpha(x)
\]
converges absolutely everywhere by similar arguments as in the proof of Lemma~\ref{lem:NUM-multi}. Applying Fubini's theorem to interchange the sum and the integral (the application can be justified by similar arguments as in the proof of Lemma~\ref{lem:NUM-multi}), we obtain
\[
\begin{aligned}
\frac{[f_g^+(x)]_l}{p_{\theta_0}(x)}-[m_{0,g}^+]_l
&=\int (\theta_l-\theta_{0,l})\left(\frac{p_\theta(x)}{p_{\theta_0}(x)}-1\right)\dif g(\theta)\\
&=\int \left(\sum_{k=1}^\infty\sum_{|\alpha|=k}
\frac{(\theta_l-\theta_{0,l})(\theta-\theta_0)^\alpha}{\alpha!}\,q_\alpha(x)\right)\dif g(\theta)\\
&=\sum_{k=1}^\infty\sum_{|\alpha|=k}\frac{[m_{\alpha,g}^+]_l}{\alpha!}\,q_\alpha(x).
\end{aligned}
\]
Rearranging and noting that the \(k=0\) term corresponds to \(m_{0,g}^+\) yields the desired expansion.
\end{proof}

\begin{proposition}\label{pro:postmean}
Assume \ref{(GP)} and suppose that $g_0$ is finitely discrete. Then there exist a measurable function \(h\in L_2(f_{g_0}\dif\mu)\) such that
$$
\big\|\bE_{\hat g}[\theta| x]-\bE_{g_0}[\theta| x]\big\|
\le h(x)\,\Delta_{\hat g},
$$
where $\|\cdot\|$ denotes the Euclidean norm on $\bR^d$, and \(\Delta_{\hat g}\) is defined in Lemma~\ref{lem:MCL-multi}. Moreover, if $b\in\{0,d\}$, then
$$
h(x)<\infty \qquad \text{for every admissible outcome } x.
$$
\end{proposition}

\begin{proof}[Proof of Proposition~\ref{pro:postmean}] Define
\[
\mathrm{Err}_{\hat g}(x)\coloneqq \big\|\bE_{\hat g}[\theta| x]-\bE_{g_0}[\theta| x]\big\|.
\]
By definition,
\[
\|f_{\hat g}^+(x)\|
\le \int \|\theta-\theta_0\|\,p_\theta(x)\,\dif\hat g(\theta)
\le M \int p_\theta(x)\,\dif\hat g(\theta)
= M f_{\hat g}(x).
\]
Applying this bound in the second inequality, and invoking Lemma~\ref{lem:NUM-multi} and Lemma~\ref{lem:expansion2} in the third, we obtain, for $x \in \mathrm{supp}(f_{g_0})$,
\begin{equation*}
\begin{aligned}
\mathrm{Err}_{\hat g}(x)
&\le
\Big\|\frac{f_{\hat g}^+(x)}{f_{\hat g}(x)}-\frac{f_{\hat g}^+(x)}{f_{g_0}(x)}\Big\|
+
\Big\|\frac{f_{\hat g}^+(x)}{f_{g_0}(x)}-\frac{f_{g_0}^+(x)}{f_{g_0}(x)}\Big\|\\
&\le
M\,\frac{|f_{\hat g}(x)-f_{g_0}(x)|}{f_{g_0}(x)}
+
\frac{\|f_{\hat g}^+(x)-f_{g_0}^+(x)\|}{f_{g_0}(x)}\\
&\le
\sqrt{C_0}\,\sqrt{\frac{p_{\theta_0}(x)}{f_{g_0}(x)}}
\bigg(
M\Big|\sum_{k=1}^\infty\sum_{|\alpha|=k}\frac{m_{\alpha,\hat g}-m_{\alpha,g_0}}{\alpha!}\,q_\alpha(x)\Big|
+
\Big\|\sum_{k=0}^\infty\sum_{|\alpha|=k}\frac{m_{\alpha,\hat g}^+-m_{\alpha,g_0}^+}{\alpha!}\,q_\alpha(x)\Big\|
\bigg),
\end{aligned}
\end{equation*}
where \(C_0\) is the same quantity as in Lemma~\ref{lem:DEN-multi}.

Recall the definition of $a_\alpha$ from Proposition~\ref{pro:polynomial} and the identity in~\eqref{eq:Frob}. Note further that
\begin{equation} \label{eq:sumb}
\begin{aligned}
\sum_{|\alpha|=k}\frac{\|m_{\alpha,\hat g}^+-m_{\alpha,g_0}^+\|^2}{\alpha!}
&=
\sum_{|\alpha|=k}\left\|
\frac{m_{\alpha,\hat g}^+-m_{\alpha,g_0}^+}{\sqrt{\alpha!(|\alpha|+1)}}\,
\sqrt{|\alpha|+1}
\right\|^2\\
&=
(k+1)\sum_{|\alpha|=k}
\left\|
\frac{m_{\alpha,\hat g}^+-m_{\alpha,g_0}^+}{\sqrt{\alpha!(|\alpha|+1)}}
\right\|^2\\
&\le
(k+1)d\sum_{|\alpha|=k+1}\frac{(m_{\alpha,\hat g}-m_{\alpha,g_0})^2}{\alpha!},
\end{aligned}
\end{equation}
where for the last line we note that $\|m_{\alpha,\hat g}^+-m_{\alpha,g_0}^+\|^2 = \sum_{l=1}^d (m_{\alpha+e_l,\hat g}-m_{\alpha+e_l,g_0})^2$ where $e_l$ is the $l$'th canonical unit vector in $\bR^d$ and that $(\alpha+e_l)! = \alpha!(\alpha_l+1) \le \alpha!(|\alpha|+1)$.

Next, applying the Cauchy--Schwarz inequality yields 
\[
\begin{aligned}
\mathrm{Err}_{\hat g}(x)
&\le
M \sqrt{C_0}
\sqrt{\sum_{k=1}^\infty\sum_{|\alpha|=k}\frac{k^2(k+1)^d a_\alpha (m_{\alpha,\hat g}-m_{\alpha,g_0})^2}{\alpha!}}\; S(x)\\
&\quad+
\sqrt{C_0}
\sqrt{\sum_{k=0}^\infty\sum_{|\alpha|=k}\frac{(k\lor 1)^2(k+1)^d a_\alpha \|m_{\alpha,\hat g}^+-m_{\alpha,g_0}^+\|^2}{\alpha!}}\; S(x)\\
&\le
M \sqrt{C_0}
\sqrt{\sum_{k=1}^\infty\Big(\sup_{|\alpha|=k} a_\alpha\Big)\frac{k^2(k+1)^d\|m_{k,\hat g}-m_{k,g_0}\|_F^2}{k!}}\; S(x)\\
&\quad+
\sqrt{C_0}
\sqrt{\sum_{k=0}^\infty\Big(\sup_{|\alpha|=k} a_\alpha\Big)\frac{(k\lor 1)^2(k+1)^{d+1}d\,\|m_{k+1,\hat g}-m_{k+1,g_0}\|_F^2}{(k+1)!}}\; S(x)\\
&\le
(M+\sqrt d)\,\sqrt{C_0}
\sqrt{\sum_{k=1}^\infty\Big(\sup_{|\alpha|\in\{k-1,k\}} a_\alpha\Big)\frac{k^2(k+1)^{d+1}\|m_{k,\hat g}-m_{k,g_0}\|_F^2}{k!}}\; S(x),
\end{aligned}
\]
where
\[
S(x)\coloneqq
\sqrt{\sum_{k=0}^\infty\sum_{|\alpha|=k}
\frac{q_\alpha^2(x)\,p_{\theta_0}(x)}
{(k\lor 1)^2(k+1)^d a_\alpha \alpha!\,f_{g_0}(x)}}.
\]
Here, we use the identities in~\eqref{eq:Frob} and~\eqref{eq:sumb} for the second inequality.
Furthermore, \(S\in L_2(f_{g_0}\dif\mu)\) by Proposition~\ref{pro:polynomial}(i).

Finally, we bound the moment-difference term using Lemma~\ref{lem:MCL-multi}:
\[
\begin{aligned}
&\sum_{k=1}^\infty\Big(\sup_{|\alpha|\in\{k-1,k\}}a_\alpha\Big)
\frac{k^2(k+1)^{d+1}\|m_{k,\hat g}-m_{k,g_0}\|_F^2}{k!}\\
&\le
\sum_{k=1}^\infty\Big(\sup_{|\alpha|\in\{k-1,k\}}a_\alpha\Big)
\frac{k^2(k+1)^{d+1}d^k\|m_{k,\hat g}-m_{k,g_0}\|_2^2}{k!}\\
&\le
C_2\,\Delta_{\hat g}^2,
\end{aligned}
\]
and
\[
\begin{aligned}
C_2
&\coloneqq
\Bigg[
\sum_{k=1}^{2J}\Big(\sup_{|\alpha|\in\{k-1,k\}}a_\alpha\Big)\frac{k^2(k+1)^{d+1}d^k}{k!}\\
&\hspace{4.7em}+
\sum_{k=2J+1}^\infty\Big(\sup_{|\alpha|\in\{k-1,k\}}a_\alpha\Big)\frac{k^4(k+1)^{d+1}d^k(M+1)^{4Jk}}{k!}
\Bigg]
<\infty.
\end{aligned}
\]
To summarize, we have established the desired bound
\[
\mathrm{Err}_{\hat g}(x)\le (M+\sqrt d)\sqrt{C_0C_2}\,S(x)\,\Delta_{\hat g}.
\]

Moreover, in the pure Poisson case (\(b=0\)), \(S(x)\) is finite for every admissible outcome \(x\) because $S \in L_2(f_{g_0}\dif\mu)$ and $\mu$ is the counting measure on $\bN$. In the pure Gaussian case (\(b=d\)), each \(q_\alpha\) is a product of Hermite polynomials and satisfies the bound\footnote{See inequality (18.14.9) in \citet{olver2024nist}. 
} 
$$
|q_\alpha(x)|=\prod_{l=1}^d|q_{\alpha_l}(x_l)|\le\prod_{l=1}^d\sqrt{\alpha_l!\exp\Big(\frac{(x_l-\theta_{0,l})^2}{2}\Big)}\le\sqrt{\frac{\alpha!}{p_{\theta_0}(x)}}. 
$$
Consequently, for all \(x\in\bR^d\),
\[
S(x)\le
\sqrt{\sum_{k=0}^\infty \frac{1}{(k\lor 1)^2\,f_{g_0}(x)}}<\infty.
\]
\end{proof}

\bigskip

\begin{proof}[Proof of Corollary~\ref{cor:postmean}]
By the first part of Theorem~\ref{thm:MAIN} and the lower bound in Lemma~\ref{lem:DEN-multi},
\[
\Delta_{\hat g}=O_\bP\Big(\frac{1}{\sqrt n}\Big).
\]
Therefore, the parametric rates in Corollary~\ref{cor:postmean} follows from Proposition~\ref{pro:postmean}.
\end{proof}

\subsubsection{Proof of Corollary~\ref{cor:wasserstein}}

To streamline the derivation, we adapt Lemma 3.1 in \citet{doss2023optimal} into the following form. 
\begin{lemma}\label{lem:W1}
Fix $g_0$ being a $J$-atomic distribution on $\bR^d$. For any $K$-atomic distribution $g$ on $\bR^d$, there exists $c_g\in\mathbb{S}^{d-1}$ such that, for all $\theta_1,\theta_2\in\mathrm{supp}(g_0)$, 
$$
\|\theta_1-\theta_2\|\le 2J^2\sqrt{d}|\langle\theta_1-\theta_2,c_g\rangle|, 
$$
and
$$
\mathcal{W}_1(g,g_0)\le 2JK\sqrt{d}\inf_{(U,V)\in\Pi(g,g_0)}\bE|\langle U-V, c_g\rangle|. 
$$
where $\Pi(g,g_0)$ is the set of couplings of $g,g_0$.
\end{lemma}

\begin{proof}[Proof of Lemma~\ref{lem:W1}]  The proof follows the proof strategy of Lemma~3.1 in \cite{doss2023optimal}. We adopt a probabilistic argument. Draw $c$ from the uniform distribution on $\mathbb{S}^{d-1}$. For any $x\in\bR^d$, by inequality (2.2) in the proof of Lemma 3.1 in \cite{doss2023optimal}
$$
\bP(|c^Tx|<t\|x\|)<t\sqrt{d}. 
$$
Let $\Theta_1\coloneqq\{\theta_1-\theta_2:\theta_1,\theta_2\in\mathrm{supp}(g_0)\}$ and $\Theta_2\coloneqq\{\theta_1-\theta_2:\theta_1\in\mathrm{supp}(g_0),\theta_2\in\mathrm{supp}(g)\}$. By a union bound,
$$
\bP(|c^Tx|<t_1\|x\|,\text{ for some }x\in\Theta_1)< J^2t_1\sqrt{d}, 
$$
and similarly, 
$$
\bP(|c^Ty|<t_2\|y\|,\text{ for some }y\in\Theta_2)<JKt_2\sqrt{d}. 
$$
Choose $t_1=1/(2J^2\sqrt{d})$ and $t_2=1/(2JK\sqrt{d})$. Then
$$
\bP\big(\{|c^Tx|\ge t_1\|x\|,\text{ for any }x\in\Theta_1\}\cap\{|c^Ty|\ge t_2\|y\|,\text{ for any }y\in\Theta_2\}\big)>0. 
$$
Hence there exists at least one direction $c_g$ for which both lower bounds hold simultaneously. The stated inequalities follow exactly as in the concluding step of the proof of Lemma 3.1 in \citet{doss2023optimal}. 
\end{proof}

\bigskip

\begin{proof}[Proof of Corollary~\ref{cor:wasserstein}]
Applying Lemma~\ref{lem:W1}, there exists \(c_{\hat g}\in\mathbb{S}^{d-1}\) such that
\[
\min_{\substack{\theta_1,\theta_2\in\mathrm{supp}(g_0)\\ \theta_1\neq\theta_2}}
\big|\langle \theta_1-\theta_2,c_{\hat g}\rangle\big|
\ge
\gamma_0
\coloneqq
\min_{\substack{\theta_1,\theta_2\in\mathrm{supp}(g_0)\\ \theta_1\neq\theta_2}}
\frac{\|\theta_1-\theta_2\|}{2J^2\sqrt d}
>0.
\]
Applying Lemma~\ref{lem:W1} again yields
\[
\begin{aligned}
\mathcal{W}_1(\hat g,g_0)
&\le
2J\,|\mathrm{supp}(\hat g)|\,\sqrt d\,
\inf_{(U,V)\in\Pi(\hat g,g_0)}
\bE\Big|\langle U-\theta_0,c_{\hat g}\rangle-\langle V-\theta_0,c_{\hat g}\rangle\Big|\\
&\le
2J\,|\mathrm{supp}(\hat g)|\,\sqrt d\,C_{J,\gamma_0,M}
\Bigg(
\max_{k\in[2J]}
\Big|
\int \langle \theta-\theta_0,c_{\hat g}\rangle^k\,\dif(\hat g-g_0)(\theta)
\Big|
\Bigg)^{1/2}\\
&\le
2J\,|\mathrm{supp}(\hat g)|\,\sqrt d\,C_{J,\gamma_0,M}
\sup_{c\in\mathbb{S}^{d-1}}
\Bigg(
\max_{k\in[2J]}
\Big|
\int \langle \theta-\theta_0,c\rangle^k\,\dif(\hat g-g_0)(\theta)
\Big|
\Bigg)^{1/2}\\
&=
2J\,|\mathrm{supp}(\hat g)|\,\sqrt d\,C_{J,\gamma_0,M}\,
\Delta_{\hat g}^{1/2},
\end{aligned}
\]
where \(\Delta_{\hat g}\) is defined as in Lemma~\ref{lem:MCL-multi} and the last line uses $\int \langle \theta-\theta_0,c\rangle^k\,\dif(\hat g-g_0)(\theta) = \langle m_{k,\hat g} - m_{k,g_0},c^{\otimes k}\rangle$. Here \(C_{J,\gamma_0,M}\) is the constant in Proposition~5 of \citet{wu2020optimal}, depending only on \(J\), \(\gamma_0\), and \(M\); this proposition yields the second inequality above.

By the first part of Theorem~\ref{thm:MAIN} and the lower bound in Lemma~\ref{lem:DEN-multi},
\[
\Delta_{\hat g}=O_\bP\Big(\frac{1}{\sqrt n}\Big).
\]
This completes the proof.
\end{proof}

\subsection{Proof of Theorem~\ref{thm:diverge}}

The analysis crucially relies on an orthonormal polynomial system associated with \(g_0\).

\begin{proposition}\label{pro:ORT-multi}
Assume that \(g_0\) is supported on a compact set \(\Theta\) and is not finitely discrete. Then there exists a sequence of polynomials \(\{q_k\}_{k\in\bN}\) on \(\Theta\) such that \(q_0(\theta)\equiv 1\) and, for any \(k,k'\in\bN\),
\[
\int_\Theta q_k(\theta)\,q_{k'}(\theta)\,\dif g_0(\theta)=\mathbf{1}_{\{k=k'\}}.
\]
\end{proposition}

\begin{proof} 
Since $g_0$ is not finitely discrete, there must exist an index $l \in [d]$ such that  the $l$th marginal of $g_0$ is not finite discrete. For this $l$, the elements in $\{1,\theta_l,\theta_l^2,\dots\}$ are linearly independent in $L_2(g_0)$. The desired polynomials can now be constructed through the Gram-Schmidt process. 
\end{proof}

These polynomials provide a convenient parameterization of a nested sequence of submodels. For each nonnegative integer \(K\), define the order-\(K\) submodel by
\[
\cF^{\le K}\coloneqq\{f_g:\ g\in\cG^{\le K}\},
\]
where
\[
\cG^{\le K}\coloneqq
\left\{g\in\cG:\ \exists\,c\in\bR^K \text{ such that }
\dif g(\theta)=\left(1+\sum_{k=1}^K c_k q_k(\theta)\right)\dif g_0(\theta)\right\}.
\]
The corresponding score set admits an explicit characterization.

\begin{lemma}\label{lem:INF-multi}
Assume \ref{(ID)} and fix \(K\ge 1\). The score set corresponding to the order-\(K\) submodel \(\cG^{\le K}\) satisfies
\begin{equation}\label{eq:SCO-multi}
\cS^{\le K}\coloneqq\big\{s_{f_g}:\ g\in\cG^{\le K}\backslash g_0\big\}
=
\left\{
\frac{\sum_{k=1}^K c_k h_k}{\bigl\|\sum_{k=1}^K c_k h_k\bigr\|_{L_2(f_{g_0}\dif\mu)}}:\ c\in\bS^{K-1}
\right\},
\end{equation}
where $\bS^{K-1}$ is the unit sphere in $\bR^K$ and
\[
h_k(x)\coloneqq
\frac{\int p_\theta(x)\,q_k(\theta)\,\dif g_0(\theta)}{f_{g_0}(x)}\,\mathbf{1}_{\{f_{g_0}(x)>0\}},
\qquad k\in[K].
\]
Moreover, \(\{h_k\}_{k=1}^K\) are linearly independent elements of \(L_2(f_{g_0}\dif\mu)\).
\end{lemma}

\begin{proof}[Proof of Lemma~\ref{lem:INF-multi}]
We establish two key facts. First, for each \(k\in[K]\), the function \(h_k\) is \(f_{g_0}\dif\mu\)-square integrable, since
\[
|h_k(x)|
\le \sup_{\theta\in\Theta}|q_k(\theta)|
<\infty.
\]
Consequently, for \(\dif g_c=(1+\sum_{k=1}^K c_k q_k)\dif g_0\in\cG^{\le K}\),
\begin{equation}\label{eq:CHI-multi}
\chi(f_{g_c},f_{g_0})
=\left\|\sum_{k=1}^K c_k h_k\right\|_{L_2(f_{g_0}\dif\mu)}
<\infty.
\end{equation}

Second, for such \(g_c\) with \(c\neq 0\), we have
\[
\int_\Theta q_k(\theta)\,\dif g_c(\theta)=c_k,
\qquad k\in[K],
\]
whereas \(\int_\Theta q_k(\theta)\,\dif g_0(\theta)=0\) for all \(k\in[K]\). Hence \(g_c\neq g_0\), and by Assumption~\ref{(ID)},
\begin{equation}\label{eq:NOM-multi}
\left\|\sum_{k=1}^K c_k h_k\right\|_{L_2(f_{g_0}\dif\mu)}
=\chi(f_{g_c},f_{g_0})
>0.
\end{equation}
In particular, the functions \(\{h_k\}_{k=1}^K\) are linearly independent. Moreover,
\[
\cS^{\le K}
\subseteq
\left\{\frac{\sum_{k=1}^K c_k h_k}{\bigl\|\sum_{k=1}^K c_k h_k\bigr\|_{L_2(f_{g_0}\dif\mu)}}:\ c\in\bS^{K-1}\right\},
\]
so it remains to prove the reverse inclusion.

To this end, fix an arbitrary \(c\in\bS^{K-1}\) and define
\[
C \coloneqq \sum_{k=1}^K \sup_{\theta\in\Theta}|q_k(\theta)|.
\]
Then \(\sup_{\theta\in\Theta}\frac{1}{C}\sum_{k=1}^K|c_k q_k(\theta)|\le 1\), and since \(\int_\Theta q_k(\theta)\,\dif g_0(\theta)=0\) for \(k\in[K]\), it follows that \(g_{c/C}\) is a probability measure on \(\Theta\). Noting that the score of \(f_{g_{c/C}}\) equals
\[
\frac{\sum_{k=1}^K c_k h_k}{\bigl\|\sum_{k=1}^K c_k h_k\bigr\|_{L_2(f_{g_0}\dif\mu)}},
\]
we conclude that
\[
\left\{\frac{\sum_{k=1}^K c_k h_k}{\bigl\|\sum_{k=1}^K c_k h_k\bigr\|_{L_2(f_{g_0}\dif\mu)}}:\ c\in\bS^{K-1}\right\}
\subseteq \cS^{\le K}.
\]
This completes the proof.
\end{proof}

We are now ready to state and prove the key result, from which Theorem~\ref{thm:diverge} follows.

\begin{theorem}\label{thm:INF-multi}
Assume \ref{(ID)}. For every \(K\ge 1\), the pair \((\cF^{\le K},f_{g_0})\) satisfies Assumptions \ref{(A1)}, \ref{(A2)}, and \ref{(SS)}. Moreover, for any \(K\ge 1\), both \(n\,\chi^2(f_{\hat{g}_K},f_{g_0})\) and \(2\{\ell_n(f_{\hat{g}_K})-\ell_n(f_{g_0})\}\) converge in distribution to \(\chi^2(K)\). Here \(\hat{g}_K\) is any (approximate) MLE over \(\cG^{\le K}\) satisfying
\[
\sup_{g\in\cG^{\le K}}\ell_n(f_g)-\ell_n(f_{\hat{g}_K})=o_\bP(1).
\]
\end{theorem}

\begin{proof}[Proof of Theorem~\ref{thm:INF-multi}] 
For any \(K\ge 1\), the pair \((\cF^{\le K},f_{g_0})\) satisfies \ref{(SS)} by definition and \ref{(A1)} by \eqref{eq:CHI-multi}. We now verify \ref{(A2)} for the score set \(\cS^{\le K}\) defined in \eqref{eq:SCO-multi}. By \eqref{eq:NOM-multi} and the compactness of \(\bS^{K-1}\),
\begin{equation}\label{eq:hlinindep-multi}
\inf_{c\in\bS^{K-1}}\left\|\sum_{k=1}^K c_k h_k\right\|_{L_2(f_{g_0}\dif\mu)}>0.
\end{equation}
In view of \eqref{eq:SCO-multi}, it follows that \(\cS^{\le K}\) is, up to a uniformly bounded scaling, contained in the convex hull of the finite collection \(\{h_k:\ k\in[K]\}\subseteq L_2(f_{g_0}\dif\mu)\). Hence \(\cS^{\le K}\) is \(f_{g_0}\dif\mu\)-Donsker and admits an \(f_{g_0}\dif\mu\)-square-integrable envelope, so Assumption~\ref{(A2)} holds; see Theorem~2.10.3 of \citet{van1996weak}.

By Theorem~\ref{thm:ASY} applied to the pair \((\cF^{\le K},f_{g_0})\), both \(n\,\chi^2(f_{\hat{g}_K},f_{g_0})\) and \(2\{\ell_n(f_{\hat{g}_K})-\ell_n(f_{g_0})\}\) converge in distribution to
\[
\sup_{s\in\cS^{\le K}}\bigl[(\bG(s))_+\bigr]^2,
\]
where \(\bG\) is a centered Gaussian process indexed by \(\cS^{\le K}\) with covariance function
\[
\mathrm{Cov}\bigl(\bG(s_1),\bG(s_2)\bigr)
\coloneqq
\int s_1(x)\,s_2(x)\,f_{g_0}(x)\,\dif\mu(x),
\qquad s_1,s_2\in\cS^{\le K}.
\]
This process admits an equivalent finite-dimensional representation. Let \(Z\sim N(0,I_K)\) and define \(\Sigma\in\bR^{K\times K}\) by
\[
\Sigma_{k_1,k_2}\coloneqq \int h_{k_1}(x)\,h_{k_2}(x)\,f_{g_0}(x)\,\dif\mu(x),
\qquad 1\le k_1,k_2\le K.
\]
Then \(\Sigma\) has full rank. Indeed, if \(\Sigma\) were singular, there would exist \(a\in\bS^{K-1}\) such that
\[
0=a^\top \Sigma a
=\sum_{i,j=1}^K a_i a_j \int h_i(x)\,h_j(x)\,f_{g_0}(x)\,\dif\mu(x)
=\int\Big(\sum_{j=1}^K a_j h_j(x)\Big)^2 f_{g_0}(x)\,\dif\mu(x),
\]
contradicting \eqref{eq:hlinindep-multi}.

By Lemma~\ref{lem:INF-multi}, for every \(s\in\cS^{\le K}\) there exists \(c=c(s)\in\bS^{K-1}\) such that
\[
s(x)=s_c(x)\coloneqq
\frac{\sum_{k=1}^K c_k h_k(x)}
{\bigl\|\sum_{k=1}^K c_k h_k\bigr\|_{L_2(f_{g_0}\dif\mu)}}.
\]
Define a centered Gaussian process \(\widetilde\bG\) on \(\bS^{K-1}\) by
\[
\widetilde\bG(c)\coloneqq \frac{c^\top \Sigma^{1/2}Z}{\sqrt{c^\top \Sigma c}},
\qquad c\in\bS^{K-1},
\]
where \(Z\sim N(0,I_K)\). It is straightforward to verify that, for any \(c,c'\in\bS^{K-1}\),
\[
\mathrm{Cov}\bigl(\widetilde\bG(c),\widetilde\bG(c')\bigr)
=
\mathrm{Cov}\bigl(\bG(s_c),\bG(s_{c'})\bigr).
\]
Consequently,
\[
\sup_{s\in\cS^{\le K}}\bigl[(\bG(s))_+\bigr]^2
\ \stackrel{\mathcal D}{=}\ 
\sup_{c\in\bS^{K-1}}\bigl[(\widetilde\bG(c))_+\bigr]^2.
\]
A direct calculation then yields the chi-square limit:
\[
\begin{aligned}
\sup_{c\in\bS^{K-1}}\bigl[(\widetilde\bG(c))_+\bigr]^2
&=\left(\sup_{c\in\bS^{K-1}} \frac{c^\top \Sigma^{1/2}Z}{\sqrt{c^\top \Sigma c}}\right)^2\\
&=\left(\sup_{u\in\bS^{K-1}} u^\top Z\right)^2
= Z^\top Z,
\end{aligned}
\]
where we used that one of \(\widetilde\bG(c)\) and \(\widetilde\bG(-c)\) is always nonnegative, and that \(\Sigma\) has full rank so that
\[
\left\{\frac{\Sigma^{1/2}c}{\sqrt{c^\top \Sigma c}}:\ c\in\bS^{K-1}\right\}=\bS^{K-1}.
\]
\end{proof}

\bigskip

\begin{proof}[Proof of Theorem~\ref{thm:diverge}] 
Theorem~\ref{thm:diverge} is now a direct consequence of Theorem~\ref{thm:INF-multi}. To prove divergence of the likelihood ratio statistic, fix \(M>0\) and \(K\ge 1\). By the inclusion \(\cF^{\le K}\subseteq \cF\) and the Portmanteau theorem,
\[
\begin{aligned}
\liminf_{n\to\infty}\bP\Big(\sup_{f\in\cF}\ell_n(f)-\ell_n(f_{g_0})>M\Big)
&\ge
\liminf_{n\to\infty}\bP\Big(\sup_{f\in\cF^{\le K}}\ell_n(f)-\ell_n(f_{g_0})>M\Big)\\
&\ge
\bP(\chi^2(K)>2M),
\end{aligned}
\]
where the last inequality uses that
\(
2\Big\{\sup_{f\in\cF^{\le K}}\ell_n(f)-\ell_n(f_{g_0})\Big\}
\Dkonv \chi^2(K)
\)
by Theorem~\ref{thm:INF-multi}. Since \(\bP(\chi^2(K)>2M)\to 1\) as \(K\to\infty\), it follows that for any \(M>0\),
\[
\liminf_{n\to\infty}\bP\Big(\sup_{f\in\cF}\ell_n(f)-\ell_n(f_{g_0})>M\Big)=1,
\]
and therefore \(L_n(\cG,g_0)\to\infty\) in probability.

To prove divergence of the normalized chi-square risk, fix \(M,c>0\) and \(K\ge 1\), and let \(\hat g_K\) be the (approximate) MLE from Theorem~\ref{thm:INF-multi}. We have
\[
\bP\Big(\sup_{g\in\cG_n(c)} n\chi^2(f_g,f_{g_0})>M\Big)
\ge
\bP\Big(n\chi^2(f_{\hat g_K},f_{g_0})>M,\ \hat g_K\in\cG_n(c)\Big).
\]
Using \(\bP(A\cap B)\ge \bP(A)-\bP(B^c)\) with \(A=\{n\chi^2(f_{\hat g_K},f_{g_0})>M\}\) and \(B=\{\hat g_K\in\cG_n(c)\}\), it follows that
\[
\bP\Big(\sup_{g\in\cG_n(c)} n\chi^2(f_g,f_{g_0})>M\Big)
\ge
\bP\Big(n\chi^2(f_{\hat g_K},f_{g_0})>M\Big)
-\bP\Big(\ell_n(f_{\hat g_K})-\ell_n(f_{g_0})\le c\Big).
\]
Taking \(\liminf\) on both sides and applying the Portmanteau theorem together with Theorem~\ref{thm:INF-multi}, we obtain
\[
\liminf_{n\to\infty}\bP\Big(\sup_{g\in\cG_n(c)} n\chi^2(f_g,f_{g_0})>M\Big)\ge
\bP\big(\chi^2(K)>M\big)
-\bP\big(\chi^2(K)\le 2c\big).
\]
Since the right-hand side can be made arbitrarily close to \(1\) by choosing \(K\) sufficiently large, the proof is complete.
\end{proof}

\subsection{Proof of Proposition~\ref{pro:divergence}}

\begin{proof}[Proof of Proposition~\ref{pro:divergence} (the Gaussian case)] 
We first treat the Gaussian setting and assume \ref{(GP)} with $d=b=1$. Without loss of generality, we take $g_0=\delta_0$, so that $f_{g_0}=p_0$. Following Hartigan's idea, we restrict attention to the two-point mixture subfamily
\[
f_{\theta,t}(x)\coloneqq (1-t)p_0(x)+t\,p_\theta(x),\qquad t\in[0,1].
\]
Fix $K\ge 1$. Choose distinct nonzero parameters $\theta_1,\dots,\theta_K$, and define
\[
\cF^{K}\coloneqq\bigl\{f_{\theta,t}:\ \theta\in\{\theta_1,\dots,\theta_K\},\ t\in[0,1]\bigr\}.
\]
The associated score set is
\[
\cS^K
\coloneqq \{s_f:\ f\in \cF^K\backslash f_{g_0}\}
=\{s_k:\ k\in[K]\},
\qquad 
s_k\coloneqq\frac{p_{\theta_k}/p_0-1}{\chi(p_{\theta_k},p_0)}.
\]
It is immediate that the pair $(\cF^{K},f_{g_0})$ satisfies Assumptions \ref{(A1)}, \ref{(A2)}, and \ref{(SS)}.

Let $\hat{f}_K$ be any (approximate) MLE over $\cF^{K}$ such that
\[
\sup_{f\in\cF^{K}}\ell_n(f)-\ell_n(\hat{f}_K)=o_\bP(1).
\]
By Theorem~\ref{thm:ASY} applied with the pair $(\cF^{K},f_{g_0})$, both $n\,\chi^2(\hat{f}_K,f_{g_0})$ and $2\{\ell_n(\hat{f}_K)-\ell_n(f_{g_0})\}$ converge in distribution to
\begin{equation}\label{eq:limit}
\max_{k\in [K]}\bigl[(\bG_k)_+\bigr]^2
=\Bigl[\bigl(\max_{k\in [K]}\bG_k\bigr)_+\Bigr]^2,
\end{equation}
where $\{\bG_k\}_{k\in[K]}$ are mean-zero, unit-variance Gaussian random variables with, for any $k_1,k_2\in[K]$,
\[
\begin{aligned}
\mathrm{Cov}(\bG_{k_1},\bG_{k_2})
&\coloneqq
\int s_{k_1}(x)\,s_{k_2}(x)\,f_{g_0}(x)\,\dif\mu(x)\\
&=\frac{\exp(\theta_{k_1}\theta_{k_2})-1}{\sqrt{(\exp(\theta_{k_1}^2)-1)(\exp(\theta_{k_2}^2)-1)}}.
\end{aligned}
\]
Without loss of generality, assume that $\Theta$ is unbounded above. Fix $x>0$. For any $\eps > 0$ there exist $\theta_1,\dots,\theta_K>0$ such that $|\mathrm{Cov}(\bG_{k_1},\bG_{k_2})| \leq\eps$ for all $k_1\neq k_2 \in [K]$. The function $\Sigma \mapsto \bP(\max_{k\in[K]} W_k \leq x)$, where $W \sim N(0,\Sigma)$, is continuous in the point $\Sigma = I_{K\times K}$, so we can pick $\eps$ small enough so that
\[
\bP\Bigl(\max_{k\in[K]}\bG_k\le x\Bigr)\le 2\,\Phi(x)^K,
\]
where $\Phi$ denotes the standard normal CDF. Fix $\delta > 0$. Since such a construction is possible for any $K$, we can construct $\theta_1,\dots,\theta_K$ in such a way that $\max_{k\in[K]}\bG_k > x$ with probability at least $1-\delta$. Following the arguments in the proof of Theorem~\ref{thm:diverge}, we find that, for any $0<M,2c\le x^2$,
\begin{align*}
&\liminf_{n\to\infty}\bP\Big(\sup_{g\in\cG_n(c)} n\chi^2(f_g,f_{g_0})>M\Big) 
\\
&\ge \liminf_{n\to\infty}
\Big\{
\bP\Big(n\chi^2(\hat f_K,f_{g_0})>M\Big)
-\bP\Big(\ell_n(\hat f_K)-\ell_n(f_{g_0})\le c\Big) \Big\}
\\
& \ge \bP\Big(\max_{k\in[K]}\bG_k>\sqrt{M}\Big)
-\bP\Big(\max_{k\in[K]}\bG_k\le \sqrt{2c}\Big)
\\
& \geq 1-2\delta. 
\end{align*}
Since $x$ can be chosen arbitrarily large and $\delta$ arbitrarily small, the proof is complete.
\end{proof}

\begin{proof}[Proof of Proposition~\ref{pro:divergence} (the Poisson case)] 
We now turn to the Poisson setting and assume \ref{(GP)} with $d=1$ and $b=0$. Fixing some $\theta_0>0$, we take $g_0=\delta_{\theta_0}$, so that $f_{g_0}=p_{\theta_0}$. As before, we restrict attention to a two-point mixture subfamily, and the corresponding limiting distribution for the quantities of interest retains the form~\eqref{eq:limit}, but with a different covariance structure:
\[
\mathrm{Cov}(\bG_{k_1},\bG_{k_2})=\frac{\exp\big((\theta_{k_1}-\theta_0)(\theta_{k_2}-\theta_0)/\theta_0\big)-1}{\sqrt{\big[\exp\big((\theta_{k_1}-\theta_0)^2/\theta_0\big)-1\big]\big[\exp\big((\theta_{k_2}-\theta_0)^2/\theta_0\big)-1\big]}}.
\]
If $\Theta$ is unbounded, we may again choose $\theta_1,\dots,\theta_K$ sufficiently separated so that the pairwise correlations are arbitrarily small, and the same argument as in the Gaussian case applies.
\end{proof}

\subsection{Proof of Theorem~\ref{thm:converge2}}
We apply Lemma~\ref{lem:CVR} and Theorem~\ref{thm:ASY}. Assumption~\ref{(SS)} holds by definition. Since \(f_{g_0}\) is fully supported on \(\{1,\dots,K\}\), Assumption~\ref{(A1)} is immediate: for any \(f\in\cF\backslash f_{g_0}\),
\[
\chi^2(f,f_{g_0})
=\sum_{k=1}^K f_{g_0}(k)\left(\frac{f(k)}{f_{g_0}(k)}-1\right)^2
\in(0,\infty).
\]

It remains to verify Assumption~\ref{(A2)}. Any \(f\in\cF\) can be written as
\[
f(\cdot)
=
\left(1+\sum_{k=1}^K \frac{f(k)-f_{g_0}(k)}{f_{g_0}(k)}\,h_k(\cdot)\right) f_{g_0}(\cdot),
\]
where \(h_k(k')\coloneqq \mathbf{1}_{\{k'=k\}}\) for \(1\le k,k'\le K\). Consequently,
\[
\cS
\subseteq
\left\{
\frac{\sum_{k=1}^K c_k h_k}{\bigl\|\sum_{k=1}^K c_k h_k\bigr\|_{L_2(f_{g_0}\dif\mu)}}:\ c\in\bS^{K-1}
\right\}.
\]
Moreover, since \(f_{g_0}\) is fully supported,
\[
\inf_{c\in\bS^{K-1}}
\left\|\sum_{k=1}^K c_k h_k\right\|_{L_2(f_{g_0}\dif\mu)}
=
\min_{k\in[K]}\sqrt{f_{g_0}(k)}
>0.
\]
Thus, 
\[
\cS \subseteq \frac{1}{\min_{k\in[K]}\sqrt{f_{g_0}(k)}}
\left\{
\sum_{k=1}^K c_k h_k: \ c\in\bS^{K-1}
\right\},
\]
where the right-hand side is a scaled version of the convex hull of the finite collection \(\{h_k:\ k\in[K]\}\subseteq L_2(f_{g_0}\dif\mu)\). 
It follows that \(\cS\) is \(f_{g_0}\dif\mu\)-Donsker by Theorem 2.10.3 in \cite{van1996weak} and admits an \(f_{g_0}\dif\mu\)-square-integrable envelope, verifying Assumption~\ref{(A2)}. \hfill\(\Box\)

\subsection{Proofs of two key lemmas underlying the parametric phenomena}

\begin{proof}[Proof of Lemma~\ref{lem:CVR}] The proof utilizes ideas from~\cite{gassiat2002likelihood} but makes adjustments accounting for the fact that we deal with approximate NPMLE. Using the inequality \(\log(1+x)\le x-x_-^2/2\) for \(x\in(-1,\infty)\), where \(x_-\coloneqq\max\{-x,0\}\), we have, for any \(f\in\cF_n\backslash f_0\),
\[
\begin{aligned}
-c_n \le \ell_n(f)-\ell_n(f_0)
&=\sum_{i=1}^n\log\Big(1+\chi(f,f_0)s_f(X_i)\Big)\\
&\le \chi(f,f_0)\sum_{i=1}^n s_f(X_i)
-\frac{1}{2}\chi^2(f,f_0)\sum_{i=1}^n\big[(s_f(X_i))_-\big]^2\\
&=\sqrt{n}\chi(f,f_0)\cdot \frac{1}{\sqrt{n}}\sum_{i=1}^n s_f(X_i)
-\frac{1}{2}n\chi^2(f,f_0)\cdot \frac{1}{n}\sum_{i=1}^n\big[(s_f(X_i))_-\big]^2,
\end{aligned}
\]
where we used the fact that \(\chi(f,f_0)s_f(X_i)=f(X_i)/f_0(X_i)-1>-1\). Viewing the last display as a quadratic inequality in \(\sqrt{n}\chi(f,f_0)\), and recalling that we assumed $c_n \geq 0$, we obtain
\[
\begin{aligned}
\sup_{f\in\cF_n}\sqrt{n}\chi(f,f_0)
&\le \sup_{f\in\cF_n\backslash f_0}
\frac{\frac{1}{\sqrt{n}}\sum_{i=1}^n s_f(X_i)
+\sqrt{\Big(\frac{1}{\sqrt{n}}\sum_{i=1}^n s_f(X_i)\Big)^2
+2c_n\frac{1}{n}\sum_{i=1}^n[(s_f(X_i))_-]^2}}
{\frac{1}{n}\sum_{i=1}^n[(s_f(X_i))_-]^2}\\
&\le
\frac{\sup_{s\in\cS}\Big(\frac{1}{\sqrt{n}}\sum_{i=1}^n s(X_i)
+\Big|\frac{1}{\sqrt{n}}\sum_{i=1}^n s(X_i)\Big|
+\sqrt{2c_n\frac{1}{n}\sum_{i=1}^n[(s(X_i))_-]^2}\Big)}
{\inf_{s\in\cS}\frac{1}{n}\sum_{i=1}^n[(s(X_i))_-]^2}.
\end{aligned}
\]
The first two terms in the numerator are \(O_\bP(1)\) by the Donsker assumption in \ref{(A2)}, noting that \(s(X_i)\) are centered under \(f_0\).
For the third term and the denominator, by Example 2.10.7 and Lemma 2.10.14 of \citet{van1996weak}, the class \(\{(s_-)^2:s\in\cS\}\) is \(f_0\dif\mu\)-Glivenko--Cantelli in Probability. Moreover, we must have
\[
\inf_{s\in\cS}\int (s_-)^2\,f_0\,\dif\mu>0.
\]
Otherwise, there would exist a sequence \(\{s_n\}_{n\in\bN}\subseteq\cS\) with \(\int [(s_n)_-]^2 f_0\,\dif\mu\to0\). Since
\(\int (s_n)_+ f_0\,\dif\mu-\int (s_n)_- f_0\,\dif\mu=\int s_n f_0\,\dif\mu=0\),
it follows that \(s_n\to0\) in \(L_1(f_0\dif\mu)\). The square-integrable envelope assumption in \ref{(A2)} then implies \(s_n\to0\) in \(L_2(f_0\dif\mu)\), contradicting \(\int s_n^2 f_0\,\dif\mu=1\).
As a result, the denominator is bounded away from zero in probability. Combining these observations yields \eqref{eq:CVR}.
\end{proof}

\begin{proof}[Proof of Lemma~\ref{lem:MCL}]
Define \(M\coloneqq\sup_{\theta\in\Theta}|\theta-\theta_0|\) and \(\Delta_g\coloneqq\max_{j\in[2J]}|m_{j,g}-m_{j,g_0}|\). We will prove the slightly stronger bound
\begin{equation}\label{eq:MCL-tmp}
|m_{k,g}-m_{k,g_0}|\le (k-2J)(M+1)^{2J(k-2J)+1}\Delta_g. 
\end{equation}
Let \(\theta_1,\dots,\theta_J\) be the support points of \(g_0\). We will repeatedly use the representation
\begin{equation}\label{eq:help1mom}
\prod_{j=1}^J(\theta-\theta_j)^2 
= \prod_{j=1}^{2J}(\theta-\theta_0+\theta_0-\kappa_j) 
= \sum_{j=0}^{2J} (\theta-\theta_0)^{2J-j}\sum_{s \subseteq [2J];\, |s| = j} \prod_{i \in s} (\theta_0 - \kappa_i), 
\end{equation}
where \(\kappa_{2j}=\kappa_{2j-1}=\theta_j\), along with the bound
\begin{equation}\label{eq:help2mom}
\Big| \sum_{s \subseteq [2J]:\, |s| = j} \prod_{i \in s} (\theta_0 - \kappa_i) \Big| 
\le \binom{2J}{j} M^{j}.
\end{equation}
We begin by noting that
\[
\begin{aligned}
\left|\int(\theta-\theta_0)^{k-2J}\prod_{j=1}^J(\theta-\theta_j)^2\,\dif(g-g_0)(\theta)\right|
&=
\left|\int(\theta-\theta_0)^{k-2J}\prod_{j=1}^J(\theta-\theta_j)^2\,\dif g(\theta)\right|
\\
&\le M^{k-2J}\int\prod_{j=1}^J(\theta-\theta_j)^2\,\dif g(\theta)
\\
&= M^{k-2J}\left|\int\prod_{j=1}^J(\theta-\theta_j)^2\,\dif(g-g_0)(\theta)\right|
\\
&\le M^{k-2J}\sum_{j=0}^{2J}\binom{2J}{j}M^{j}\left|\int(\theta-\theta_0)^{2J-j}\,\dif(g-g_0)(\theta)\right|
\\
&\le (M+1)^{2J}M^{k-2J}\Delta_g,
\end{aligned}
\]
where we used \eqref{eq:help1mom} and \eqref{eq:help2mom} in the second inequality, and the identity
\[
\sum_{j=0}^{2J}\binom{2J}{j}M^{j}=(M+1)^{2J}
\]
in the last step.

To proceed, recall that \(k>2J\) and observe that
\[
\begin{aligned}
|m_{k,g}-m_{k,g_0}|
&=\left|\int(\theta-\theta_0)^k\,\dif(g-g_0)(\theta)\right|
\\
&\le (M+1)^{2J}M^{k-2J}\Delta_g
+\left|\int\Big((\theta-\theta_0)^k-(\theta-\theta_0)^{k-2J}\prod_{j=1}^J(\theta-\theta_j)^2\Big)\,\dif(g-g_0)(\theta)\right|
\\
&\le (M+1)^{2J}M^{k-2J}\Delta_g
+\sum_{j=1}^{2J}\binom{2J}{j}M^{j}\left|\int(\theta-\theta_0)^{k-j}\,\dif(g-g_0)(\theta)\right|
\\
&\le (M+1)^{2J}M^{k-2J}\Delta_g
+(M+1)^{2J}\sup_{j\in[2J]}|m_{k-j,g}-m_{k-j,g_0}|,
\end{aligned}
\]
where we again used \eqref{eq:help1mom} and \eqref{eq:help2mom} in the second inequality.

We now prove \eqref{eq:MCL-tmp} by induction. For the base case \(k=2J+1\), the preceding bound gives
\[
\begin{aligned}
|m_{2J+1,g}-m_{2J+1,g_0}|
&\le (M+1)^{2J}M\Delta_g+(M+1)^{2J}\Delta_g\\
&=(M+1)^{2J+1}\Delta_g.
\end{aligned}
\]
Assume \eqref{eq:MCL-tmp} holds for all integers from \(2J+1\) up to \(k-1\). Then,
\[
\begin{aligned}
|m_{k,g}-m_{k,g_0}|
&\le (M+1)^{2J}M^{k-2J}\Delta_g
+(k-1-2J)(M+1)^{2J(k-2J)+1}\Delta_g\\
&\le (k-2J)(M+1)^{2J(k-2J)+1}\Delta_g,
\end{aligned}
\]
where in the last inequality we used
\[
2J(k-2J)+1-k=(2J-1)k-4J^2+1 \ge (2J-1)(2J+1)-4J^2+1=0.
\]
This completes the proof.
\end{proof}

\subsection{Proof of Theorem~\ref{thm:ASY}} \label{sec:proofASY}

\begin{lemma}\label{lem:LOW}
Under \ref{(SS)} and \ref{(A1)}, for any \(s\in\cS\),
\begin{equation}\label{eq:LOW}
\sup_{f\in\cF}\ell_n(f)-\ell_n(f_0)\ge \frac{1}{2}\bigl[(\bG_n(s))_+\bigr]^2+o_{\bP}(1).
\end{equation}
\end{lemma}

\begin{proof} By \ref{(SS)}, for any \(s=s_f\in\cS\) there is an associated submodel \(\{f_t\}_{t\in[0,\tau]}\subseteq\cF\) given by
\[
f_t\coloneqq\left(1-\frac{t}{\chi(f,f_0)}\right)f_0+\frac{t}{\chi(f,f_0)}f,
\]
where \(\tau\coloneqq \chi(f,f_0)>0\) by \ref{(A1)}. Since
\[
\sup_{f\in\cF}\ell_n(f)-\ell_n(f_0)\ge \sup_{t\in[0,\tau]}\ell_n(f_t)-\ell_n(f_0),
\]
it suffices to lower bound the right-hand side.

Set \(\hat t_n\coloneqq (\bG_n(s))_+\). By square integrability of \(s\), we have \(\hat t_n=O_\bP(1)\), hence \(\hat t_n/\sqrt{n}=o_\bP(1)\). In particular,
\begin{equation}\label{eq:hattn}
\bP\big(\hat t_n/\sqrt{n}\in[0,\tau]\big)\to 1.
\end{equation}
By definition and a Taylor expansion,
\[
\begin{aligned}
\sup_{t\in[0,\tau]}\ell_n(f_t)-\ell_n(f_0)
&\ge \ell_n\!\left(f_{\hat t_n/\sqrt{n}}\right)-\ell_n(f_0)+o_\bP(1)\\
&=\sum_{i=1}^n \log\left(1+\frac{\hat t_n}{\sqrt{n}}\,s(X_i)\right)+o_\bP(1)\\
&=\frac{\hat t_n}{\sqrt{n}}\sum_{i=1}^n s(X_i)
-\frac{\hat t_n^2}{2n}\sum_{i=1}^n s^2(X_i)
+\frac{\hat t_n^2}{n}\sum_{i=1}^n s^2(X_i)\,
R\!\left(\frac{\hat t_n}{\sqrt{n}}\,s(X_i)\right)
+o_\bP(1)\\
&=\frac{1}{2}\bigl[(\bG_n(s))_+\bigr]^2
+\frac{\hat t_n^2}{n}\sum_{i=1}^n s^2(X_i)\,
R\!\left(\frac{\hat t_n}{\sqrt{n}}\,s(X_i)\right)
+o_\bP(1),
\end{aligned}
\]
where \(R\) is a deterministic remainder function satisfying \(R(x)\to 0\) as \(x\to 0\).
In the first line we used \eqref{eq:hattn}. In the second line we used that \(s_{f_t}=s_f\) for all \(t\in[0,\tau]\) and \(\chi(f_t,f_0)=t\). In the last line we used \(\int s^2 f_0\,\dif\mu=1\) together with the law of large numbers.

Finally, fix \(\varepsilon>0\). By dominated convergence,
\[
\bP\left(\frac{1}{\sqrt{n}}\max_{i\in[n]}|s(X_i)|\ge\varepsilon\right)
\le n\,\bP\big(s^2(X_1)\ge n\varepsilon^2\big)
\le \varepsilon^{-2}\int_{\{x:\ s^2(x)>n\varepsilon^2\}} s^2(x)\,f_0(x)\,\dif\mu(x)
=o(1).
\]
Therefore, the argument of \(R\big(\frac{\hat t_n}{\sqrt{n}}s(X_i)\big)\) is \(o_\bP(1)\) uniformly in \(i\), and
\[
\left|
\frac{\hat t_n^2}{n}\sum_{i=1}^n s^2(X_i)\,
R\!\left(\frac{\hat t_n}{\sqrt{n}}\,s(X_i)\right)
\right|
\le
\hat t_n^2\left(\frac{1}{n}\sum_{i=1}^n s^2(X_i)\right)
\max_{i\in[n]}\left|R\!\left(\frac{\hat t_n}{\sqrt{n}}\,s(X_i)\right)\right|
=o_\bP(1).
\]
Combining the displays yields \eqref{eq:LOW}.
\end{proof}

\bigskip

\begin{proof}[Proof of Theorem \ref{thm:ASY}] 
We start with the ``$\ge$'' direction of the second equality in \eqref{eq:ASY}. By the Donsker assumption \ref{(A2)} and the discussion in Section 2.1.2 of \cite{van1996weak}, the class $\cS$ is totally bounded in $L_2(f_0\,\dif\mu)$. Therefore, for any $m>0$, there exists a finite $1/m$-net of $\cS$ with respect to this norm; denote one such net by $\cS_m$. Throughout the proof, we abbreviate $[(\bG_n(s))_+]^2 = (\bG_n(s))_+^2$ to lighten the notation.

Fix an arbitrary $\eps>0$. By the union bound,
\[
\bP\left(\sup_{f\in\cF}\ell_n(f)-\ell_n(f_0)\le\frac{1}{2}\max_{s\in\cS_m}(\bG_n(s))_+^2-\eps\right)
\le
\sum_{s\in\cS_m}\bP\left(\sup_{f\in\cF}\ell_n(f)-\ell_n(f_0)\le\frac{1}{2}(\bG_n(s))_+^2-\eps\right).
\]
By Lemma \ref{lem:LOW}, the right-hand side converges to zero as $n\to\infty$. To conclude, consider the decomposition
\[
\begin{aligned}
\bP\left(\sup_{f\in\cF}\ell_n(f)-\ell_n(f_0)\le\frac{1}{2}\sup_{s\in\cS}(\bG_n(s))_+^2-\eps\right)
&\le
\bP\left(\sup_{f\in\cF}\ell_n(f)-\ell_n(f_0)\le\frac{1}{2}\max_{s\in\cS_m}(\bG_n(s))_+^2-\frac{\eps}{2}\right)
\\
&\quad+
\bP\left(\frac{1}{2}\max_{s\in\cS_m}(\bG_n(s))_+^2\le\frac{1}{2}\sup_{s\in\cS}(\bG_n(s))_+^2-\frac{\eps}{2}\right).
\end{aligned}
\]
The first term is handled by the preceding result for $\cS_m$. For the second term, note that
\[
\bP\left(\sup_{s\in\cS_m}(\bG_n(s))_+^2\le\sup_{s\in\cS}(\bG_n(s))_+^2-\eps\right)
\le
\bP\left(
\sup_{\substack{s_1,s_2\in\cS:\\ \|s_1-s_2\|_2\le 1/m}}
\left|(\bG_n(s_1))_+^2-(\bG_n(s_2))_+^2\right|\ge\eps
\right).
\]
Next, observe that
\[
\sup_{\substack{s_1,s_2\in\cS:\\ \|s_1-s_2\|_2\le 1/m}}
\left|(\bG_n(s_1))_+^2-(\bG_n(s_2))_+^2\right|
\le
2\left(\sup_{s\in\cS}|\bG_n(s)|\right)
\left(
\sup_{\substack{s_1,s_2\in\cS:\\ \|s_1-s_2\|_2\le 1/m}}
|\bG_n(s_1)-\bG_n(s_2)|
\right).
\]
Consequently, since the sequence $\{\bG_n\}_{n\in\bN}$ is asymptotically uniformly $L_2(f_0\,\dif\mu)$-equicontinuous in probability (see Example 1.5.10 in \cite{van1996weak}),
\[
\lim_{m\to\infty}\limsup_{n\to\infty}
\bP\left(
\sup_{\substack{s_1,s_2\in\cS:\\ \|s_1-s_2\|_2\le 1/m}}
\left|(\bG_n(s_1))_+^2-(\bG_n(s_2))_+^2\right|\ge\eps
\right)=0.
\]
This establishes the ``$\ge$'' direction of the second equality in \eqref{eq:ASY}.

To prove the ``$\le$'' direction of the second equality in~\eqref{eq:ASY}, we apply a Taylor expansion argument similar to that used in the proof of Lemma \ref{lem:LOW}. Specifically, for any $f\in\cF\backslash f_0$,
\[
\begin{aligned}
\ell_n(f)-\ell_n(f_0)
&=\sum_{i=1}^n \log\Big(1+\chi(f,f_0)s_f(X_i)\Big)\\
&=\chi(f,f_0)\sum_{i=1}^n s_f(X_i)
-\frac{1}{2}\chi^2(f,f_0)\sum_{i=1}^n s_f^2(X_i)\\
&\quad+\chi^2(f,f_0)\sum_{i=1}^n s_f^2(X_i)\,R\Big(\chi(f,f_0)s_f(X_i)\Big),
\end{aligned}
\]
where $R$ is a deterministic function satisfying $R(x)\to 0$ as $x\to 0$. Let $S$ be an $f_0\,\dif\mu$-square-integrable envelope for $\cS$. By the union bound and the dominated convergence theorem, for any fixed $\eps>0$,
\[
\begin{aligned}
\bP\left(\frac{1}{\sqrt{n}}\sup_{f\in\cF\backslash f_0}\max_{i\in[n]}|s_f(X_i)|\ge\eps\right)
&\le
n\,\bP\left(S^2(X_1)\ge n\eps^2\right)\\
&\le
\frac{1}{\eps^2}\int_{\{x:\,S^2(x)>n\eps^2\}} S^2(x)\,f_0(x)\,\dif\mu(x)
=o(1),
\end{aligned}
\]
as $n\to\infty$. Recall the notation $\cF_n = \cF_n(c_n)$ from Lemma~\ref{lem:CVR}. For the subsequent arguments, fix $c_0 > 0$ and take $c_n\equiv c_0>0$. By Lemma~\ref{lem:CVR} we have
\[
\sup_{f\in\cF_n}\chi(f,f_0)=O_{\bP}(n^{-1/2}).
\]
Thus, defining
\[
Y_n \coloneqq \Big(\sup_{f\in\cF_n}\chi(f,f_0)\Big)
\Big(\sup_{f\in\cF\backslash f_0}\max_{i\in[n]}|s_f(X_i)|\Big)
=o_{\bP}(1),
\]
we obtain
\[
\begin{aligned}
\sup_{f\in\cF_n\backslash f_0}
\left|
\chi^2(f,f_0)\sum_{i=1}^n s_f^2(X_i)\,
R\Big(\chi(f,f_0)s_f(X_i)\Big)
\right|
&\le
\Big(n\sup_{f\in\cF_n}\chi^2(f,f_0)\Big)
\left(\frac{1}{n}\sum_{i=1}^n S^2(X_i)\right)
\sup_{|x|\le Y_n}|R(x)|\\
&=o_{\bP}(1).
\end{aligned}
\]
Moreover, $\cS^2$ is $f_0\,\dif\mu$-Glivenko--Cantelli since $\cS$ is $f_0\,\dif\mu$-Donsker under \ref{(A2)}; see Lemma 2.10.14 in \cite{van1996weak}. Hence,
\[
\frac{1}{n}\sum_{i=1}^n s_f^2(X_i)=1+o_{\bP}(1),
\]
uniformly in $f\in\cF\backslash f_0$. Thus, 
\begin{equation}\label{eq:LikelihoodExpansion}
\ell_n(\hat f)-\ell_n(f_0)
=
\chi(\hat f,f_0)\sum_{i=1}^n s_{\hat f}(X_i)
-\frac{1}{2}n\chi^2(\hat f,f_0)
+o_{\bP}(1).
\end{equation}
Finally, maximizing the leading term $\chi(\hat f,f_0)\sum_{i=1}^n s_{\hat f}(X_i)
-\frac{1}{2}n\chi^2(\hat f,f_0)$ on the right-hand side over $\chi(\hat f,f_0)\ge 0$ yields the upper bound
\[
\ell_n(\hat f)-\ell_n(f_0)
\le
\frac{1}{2}\sup_{s\in\cS}(\bG_n(s))_+^2+o_{\bP}(1).
\]
Together with the lower bound established in the first part of the proof, this completes the proof for the second equality in \eqref{eq:ASY}.

To prove the first equality in \eqref{eq:ASY}, we start from the bound
\[
\begin{aligned}
\sqrt{n}\,\chi(\hat f,f_0)\,\sup_{s\in\cS}(\bG_n(s))_+
-\frac{1}{2}n\chi^2(\hat f,f_0)
&\ge
\chi(\hat f,f_0)\sum_{i=1}^n s_{\hat f}(X_i)
-\frac{1}{2}n\chi^2(\hat f,f_0)\\
&=
\ell_n(\hat f)-\ell_n(f_0)+o_{\bP}(1)\\
&=
\frac{1}{2}\sup_{s\in\cS}(\bG_n(s))_+^2+o_{\bP}(1)
\end{aligned}
\]
where the first equality follows from~\eqref{eq:LikelihoodExpansion} and the last from the second equality in~\eqref{eq:ASY} which we already proved above. Rearranging this inequality gives
\[
\Big(\sqrt{n}\chi(\hat f,f_0)-\sup_{s\in\cS}(\bG_n(s))_+\Big)^2 \le o_{\bP}(1),
\]
and hence the first equality in \eqref{eq:ASY} follows.
\end{proof}

\end{document}